\documentclass[11pt,a4paper,openany]{amsart}

\usepackage{amsmath,amsfonts}
\usepackage{latexsym}
\usepackage{amssymb}
\usepackage{color}
\usepackage{epsfig}
\usepackage{subfigure}
\usepackage{mathrsfs}
\usepackage{amscd}
\usepackage{amsthm}
\usepackage{color}
\usepackage{cite}
\usepackage{amsmath,amsfonts}
\usepackage{latexsym}
\usepackage{amssymb,leftidx}
\usepackage{extarrows}
\usepackage{overpic}
\usepackage{color}
\usepackage{color}
\usepackage{epsfig}
\usepackage{subfigure}
\usepackage{tikz}
\usetikzlibrary{positioning}
\usepackage{caption}

 \usepackage[colorlinks=true]{hyperref}
\hypersetup{urlcolor=blue, citecolor=red}

\usepackage{graphicx}
\usepackage{float}

\numberwithin{equation}{section}
\newtheorem{theorem}{Theorem}[section]
\newtheorem{proposition}[theorem]{Proposition}
\newtheorem{definition}[theorem]{Definition}
\newtheorem{lemma}[theorem]{Lemma}
\newtheorem*{theorem*}{Theorem}
\newtheorem*{lemma*}{Lemma}
\newtheorem{corollary}[theorem]{Corollary}
\newtheorem{remark}[theorem]{Remark}

\newcommand\ttheta{\tilde{\theta}}
\newcommand\tr{\tilde{r}}
\newcommand\talpha{\tilde{\alpha}}
\newcommand\tsigma{\tilde{\sigma}}

\newcommand{\R}{\mathbb{R}}

\newcommand{\Z}{{\mathbb{Z}}}
\newcommand{\LL}{{\mathcal{L}}}
\newcommand{\CC}{{\mathcal{C}}}

\begin{document}
\title[Strichartz estimates in a conical singular space]
{Strichartz estimates for the Klein-Gordon equation in a conical singular space}

\author{J. Ben-Artzi}
\address{School of Mathematics,
Cardiff University, Cardiff CF24 4AG, United Kingdom}
\email{Ben-ArtziJ@cardiff.ac.uk}

\author{F. Cacciafesta}
\address{Dipartimento di Matematica, Universit\'a degli studi di Padova, Via Trieste, 63, 35131 Padova PD, Italy}
\email{cacciafe@math.unipd.it}

\author{A. S. de Suzzoni} 
\address{CMLS, \'Ecole Polytechnique, CNRS, Universit\'e Paris- Saclay, 91128 PALAISEAU Cedex, France}
\email{ anne-sophie.de-suzzoni@polytechnique.edu}

\author{J. Zhang}
\address{Department of Mathematics, Beijing Institute of Technology, Beijing
100081, China; School of Mathematics,
Cardiff University, Cardiff CF24 4AG, United Kingdom}
\email{zhang\_junyong@bit.edu.cn; ZhangJ107@cardiff.ac.uk}

\thanks{JBA  acknowledges support from an Engineering and Physical Sciences Research Council Fellowship (EP/N020154/1). FC acknowledges support from the University of Padova STARS project ``Linear and Nonlinear Problems for the Dirac Equation" (LANPDE). ASdS acknowledges support from ESSED ANR-18-CE40-0028. JZ acknowledges support from National Natural
Science Foundation of China (11771041, 11831004) and a Marie Sk\l odowska-Curie Fellowship (790623).}

\maketitle

\begin{abstract} Consider a conical singular space $X=C(Y)=(0,\infty)_r\times Y$ with the metric $g=\mathrm{d}r^2+r^2h$, where the cross section $Y$ is
a compact $(n-1)$-dimensional closed Riemannian manifold $(Y,h)$. We study the Klein-Gordon equations with inverse-square potentials in the space $X$, proving in particular global-in-time Strichartz estimates in this setting. 
\end{abstract}

\setcounter{tocdepth}{1} 
\tableofcontents
 
\section{Introduction and main result}

Following the results obtained in \cite{ZZ1, ZZ2}, in this paper we continue the study of dispersive flows on conical singular spaces, focusing on the study of Strichartz estimates for the Klein-Gordon equation.


\subsection{The setting and motivations} The setting here is the same as \cite{ZZ1,ZZ2} in which the last author and Zheng proved the Strichartz estimates for wave and Schr\"odinger equations: let us briefly introduce it. 

Let $(Y,h)$ be a compact $(n-1)$-dimensional Riemannian manifold, consider the metric cone $X=C(Y)=(0,\infty)_r\times Y$ with $g=dr^2+r^2h$, and let $\Delta_g$ be the Friedrichs extension of the Laplace-Beltrami operator. The metric cone $X$ has a simple geometric singularity, and its metric is incomplete at $r=0$; notice that if $Y$ is the standard sphere of radius one, then $C(Y)=\mathbb{R}^{n}\setminus\{0\}$.  The spectral theory of the operator $\Delta_g$ was studied in \cite{C1,C2}. The heat kernel associated to the operator $\Delta_g$ was studied in \cite{Mo,L2}, and the Riesz transform kernel was investigated in \cite{HL,L1}.  
In the setting of exact cones $X$, in \cite{CT,CT1} the authors studied the wave diffraction, and in \cite{MW} , the propagation of singularities theory for wave equation on general setting of conic manifolds is discussed.

The Strichartz estimates for dispersive equations on conic manifolds have attracted quite some interest in recent years. The ones for the Schr\"odinger equation were proved, in the case of a flat cone, in \cite{Ford}, on
polygonal domains in \cite{BFHM}, on exterior polygonal domains in \cite{BMW}, and on the metric cone
in \cite{ZZ1}. Concerning the wave equation on cones, in \cite{BFM} the authors established the Strichartz inequalities on a flat cone of dimension two, that is, $Y=\mathbb{S}^1_\rho$. These results on the Strichartz estimates for wave and Schr\"odinger equations have been extended to general cones in \cite{ZZ1,ZZ2}. 
 \vspace{0.2cm}

In this paper, we consider the Strichartz estimates for the Klein-Gordon flow in this conical singular space. More precisely, we are interested in the study of the following Klein-Gordon equation in this framework:
\begin{equation}\label{equ:KG}
\begin{cases}
\partial_{t}^2u+\LL_V u+m^2u=F(t,z), \quad (t,z)\in I\times X, \\ u(0)=u_0(z),
~\partial_tu(0)=u_1(z),
\end{cases}
\end{equation}
where the Schr\"odinger operator 
\begin{equation}\label{LV}
\mathcal{L}_V=\Delta_g+V
\end{equation}  and $V=V_0(y) r^{-2}$ is a Hardy-type potential with $V_0(y)$ being a smooth function on the section $Y$. The inverse-square type potential is homogeneous of degree $-2$ and it is known to be at the threshold of decay in order to guarantee validity of a Strichartz estimate (see \cite{GVV}).

\begin{remark}
Beside its own independent interest, we mention the fact that the original motivation that led us to study the dynamics of the Klein-Gordon equation is the study of the Dirac equation in this conical setting. The Dirac operator $\mathcal{D}$ on $X$ was studied in \cite{Chou},
mostly from the point of view of spectral analysis. In a forthcoming work, we  intend to study the time-dependent Dirac flow by relying on the classical ``squaring trick", in order to reduce the study of the dynamics to the one of a (system-of) Klein-Gordon equations with an inverse-square potential perturbation in the same conical singular space. Nevertheless, we need to stress the fact that the geometry does not allow a straightforward and harmless application of this strategy, as indeed in order to properly define the Dirac operator on curved spaces the spin connection is needed, and therefore the corresponding Laplace operator yielded by the squaring is not the  ``scalar" one, but it is the ``spinorial" one, and this fact of course requires additional care.
\end{remark}
 \vspace{0.2cm}

 \subsection{Strichartz estimates for the Klein-Gordon equation} 
 In the flat Euclidean space, the free Klein-Gordon equation reads
 \begin{equation}\label{equ:E-KG}
\begin{cases}
\partial_{t}^2u-\Delta u+m^2u=0, \quad (t,z)\in I\times\R^n; \\ u(0)=u_0(z),
~\partial_tu(0)=u_1(z).
\end{cases}
\end{equation}
The homogenous Strichartz estimates are given by
\begin{equation*}
\begin{split}
&\|u(t,z)\|_{L^q_t(I;L^r_z(\R^n))}\lesssim \|u_0\|_{ H^s(\R^n)}+\|u_1\|_{
H^{s-1}(\R^n)},
\end{split}
\end{equation*}
where $H^s(\R^n)$ is the usual Sobolev space, and the pairs $(q,r)\in \Lambda_{s,\theta}$ with $0\leq\theta\leq1$ (the set $\Lambda_{s,\theta}$ is given in Definition \ref{ad-pair} below).
We refer the reader to the fundamental papers \cite{B, GV, KT} for these Strichartz estimates. 
Notice that in the particular case $\theta=0$, these estimates correspond to the ones for the wave equation, provided the inhomogeneous Sobolev space norms are replaced by the homogeneous ones. 
 \vspace{0.2cm}
 
In the present paper, we mainly consider the Klein-Gordon equation \eqref{equ:KG} associated with the operator $\LL_V$ in \eqref{LV}, where
$V(r,y)=r^{-2}V_0(y)$ and $V_0(y)\in\CC^\infty(Y)$ such that $\Delta_h+V_0(y)+(n-2)^2/4$ is a strictly positive operator on $L^2(Y)$. \vspace{0.2cm}

 Before stating our main result, let us introduce some notations. 

\begin{definition}\label{ad-pair}
For $0\leq\theta\leq1$  we say that a couple $(q,r)\in [2,\infty]\times [2,\infty)$ is admissible,   if $(q,r)$ satisfies
\begin{equation}\label{adm}
\frac{2}q+\frac{n-1+\theta}r\leq\frac{n-1+\theta}2,\quad (q,r, n,\theta)\neq(2,\infty,3,0).
\end{equation}
For $s\in\R$, we denote $(q,r)\in \Lambda_{s,\theta}$ if $(q,r)$ is admissible and satisfies 
\begin{equation}\label{scaling}
\frac1q+\frac {n+\theta}r=\frac {n+\theta}2-s.
\end{equation}
Let $\nu_0>0$, we define
\begin{equation}\label{Ls}
\Lambda_{s,\theta,\nu_0}=\big\{(q,r)\in\Lambda_{s,\theta}: 1/r>1/2-(1+\nu_0)/n \big\}.
\end{equation}

\end{definition}
 
Throughout this paper, pairs of conjugate indices will be written as $r, r'$, meaning that $\frac{1}r+\frac1{r'}=1$ with $1\leq r\leq\infty$. \vspace{0.2cm}

Our main result is then the following.

\begin{theorem}\label{thm:Strichartz} Assume that $(X, g)$ is a metric cone of dimension
$n\geq3$. Let $\LL_V=\Delta_g+V$ where $r^2V=:V_0(y)\in\CC^\infty(Y)$ such that $\Delta_h+V_0(y)+(n-2)^2/4$ is a strictly positive operator on $L^2(Y)$ and its smallest
eigenvalue is $\nu_0^2$ with $\nu_0>0$. Suppose
that $u$ is the solution of the Cauchy problem \eqref{equ:KG} with $m=1$ and initial data $u_0\in  H^{s}(X), u_1\in  H^{s-1}(X)$ for $s\geq0$ where $ H^s(X)={(1+\LL_V)}^{-\frac{s}2}L^2(X)$. For $0\leq \theta,\ttheta \leq 1$,  let the sets $\Lambda_{s, \theta}$ and $\Lambda_{s, \theta, \nu_0}$ be given by Definition \ref{ad-pair}. Then \vspace{0.2cm}

$\mathrm{(i)}$ if $V\equiv0$,  the Strichartz estimates
\begin{equation}\label{stri}
\begin{split}
&\|u(t,z)\|_{L^q_t(\R;L^r_z(X))}\lesssim \|u_0\|_{ H^s(X)}+\|u_1\|_{
H^{s-1}(X)}+\|F\|_{L^{\tilde{q}'}(\R;L^{\tilde{r}'}(X))}
\end{split}
\end{equation} 
hold for all $(q,r)\in \Lambda_{s,\theta}$, $(\tilde q, \tilde r)\in \Lambda_{1-s,\ttheta}$; \vspace{0.2cm}

$\mathrm{(ii)}$  if $V\not\equiv0$ and $q>2$, the Strichartz estimates 
\begin{equation}\label{V-stri}
\begin{split}
&\|u(t,z)\|_{L^q_t(\R;L^r_z(X))}\lesssim \|u_0\|_{ H^s(X)}+\|u_1\|_{
H^{s-1}(X)}+\|F\|_{L^{\tilde{q}'}(\R;L^{\tilde{r}'}(X))}
\end{split}
\end{equation}  hold for all $(q,r)\in \Lambda_{s,\theta,\nu_0}$, $(\tilde q, \tilde r)\in \Lambda_{1-s,\ttheta,\nu_0}$. Moreover, if $q=2$ and $F=0$, 
the Strichartz estimates 
\begin{equation}\label{V-stri'}
\begin{split}
\|(1+\Delta_g)^{-\frac{\alpha+\alpha_0+\sigma(r)}2}&u(t,z)\|_{L^2_t(\R;L^r_z(X))}\\&\lesssim \|u_0\|_{ H^{\frac{n+1+\theta}{2(n-1+\theta)}-\alpha-\alpha_0}(X)}+\|u_1\|_{
H^{\frac{n+1+\theta}{2(n-1+\theta)}-\alpha-\alpha_0-1}(X)}
\end{split}
\end{equation}  hold for all $(2,r)\in \Lambda_{s,\theta,\nu_0}$ where 
$$
\alpha=\alpha(\theta)=\frac1{n-1+\theta},\quad \alpha_0=\alpha(0),\quad\sigma(r)=n\left(\frac{n-3+\theta}{2(n-1+\theta)}-\frac1r\right).
$$

\end{theorem}

\begin{remark} This  is a generalization of Strichartz estimates for wave and Schr\"odinger equations proved in \cite{ZZ1,ZZ2}. The admissible pairs in $\Lambda_{s,\theta}$ match the wave ones when $\theta=0$ and the Schr\"odinger ones when $\theta=1$.
\end{remark}

\begin{remark} 
The influence of $\nu_0$ on the range of the pair $(q,r)$ is indeed nontrivial, as we will see in Proposition \ref{count}.
 \end{remark}
 
 \begin{remark} In particular, in the endpoint case $q=2$,  if one replaces the inhomogeneous Sobolev norm with the homogeneous one,  then  \eqref{V-stri'}  matches the Strichartz estimate for wave equation in \cite{BPSS}. 
Nevertheless, the unpleasant issues in \eqref{V-stri'} are due to commuting the free Laplacian $\Delta_g$ with its counterpart $\LL_V$. Even in  Euclidean space,  Strichartz estimates for Klein-Gordon \cite{LSS} with an inverse-square potential are missing at $q=2$.
 \end{remark}
 
Let us briefly comment on the strategy of the proofs. Due to the geometry, there are two obstacles that might prevent us from obtaining  dispersive estimates. The first one arises from the possible presence of conjugate points (we say $z,z'\in X$ are conjugate points, if geodesic flows can have 1-parameter families of geodesics with the same endpoints $z,z'$)  and the second  from the inverse-square potential.  Indeed, as the assumptions on the metrics $h$ on $Y$ are very general, conjugate points might appear, so the exponential map fails to be a global diffeomorphism which gives us information on geodesics. The degeneracy of the exponential map will slow down the dispersive decay estimate of the propagator, 
as illustrated in \cite{HW, HZ}. The inverse-square potential shows the same homogenous degree as the Laplacian (or, in other words, is ``scaling critical").
Therefore, it is not possible to rely on standard perturbative arguments to obtain a dispersive estimate for the half wave operator with norm $O(|t-\tau|^{-\frac{n-1}2})$ as $|t-\tau|\to \infty$ due to the influence of the negative inverse-square potential. As  is well known indeed, the analysis of dispersive estimates for scaling-critical potential perturbations of dispersive flows is quite delicate; for some details on this problem in Euclidean space, we refer to \cite{BPSS, BPST} for wave and Schr\"odinger equations and to \cite{LSS} for the Klein-Gordon equation. Fortunately, in \cite{ZZ1,ZZ2}, the authors have overcome these two issues and proved Strichartz estimates for Schr\"odinger and wave equation in an analogous setting. 
There are two key points needed to overcome the obstacles, which have been established in these papers. One consists in micro-localizing the propagator in order to separate  conjugate points: this can be achieved by studying   the micro-localized spectral measure associated to the operator $\Delta_g$ (without potentials). The other one consists in establishing a global-in-time local smoothing estimate, which is proved via a variable-separating argument: this tool allows us to deal with the perturbation provided by the inverse-square potential.\\

Our strategy is then strongly inspired by \cite{HZ, ZZ1,ZZ2, ZZ3}, but some new ideas are needed in order to deal with the difficulties arising from the Klein-Gordon equation.  
Notice indeed that the Klein-Gordon multiplier $e^{it\sqrt{1+\lambda^2}}$ behaves like the wave one for high frequencies and like the Schr\"odinger one for low frequencies, and thus we have to establish a dispersive estimate with norm $O((1+|t|)^{-n/2})$ at low frequencies and $O(2^{j(n+1+\theta)/2}(2^{-j}+|t|)^{-(n-1+\theta)/2})$ at frequencies $2^j, (j\geq0)$.
Therefore we   need to combine the strategies of \cite{HZ, ZZ1,ZZ2, ZZ3} with Keel-Tao's \cite{KT} abstract method to prove global-in-time Strichartz estimates for the ``free" Klein-Gordon equation, i.e. when the potential term vanishes, and then combine it with a local smoothing estimate for Klein-Gordon (for low and high frequencies respectively) in order to deal with the potential term. We also stress the fact that the proof of Strichartz estimates at the endpoint $q=2$ requires additional technical care, due to the lack of the Christ-Kiselev Lemma and of the usual dispersive estimates, and the issues driven by the necessity of commuting the free Laplacian $\Delta_g$ with $\LL_V$.\\

Furthermore, we point out that compared with the asymptotically conic manifold considered in \cite{HZ, ZZ3}, the metric cone here is scaling invariant (and thus automatically non-trapping) but singular at the cone tip. 
Finally we stress the fact that since our setting does not preclude the existence of conjugate points, we cannot obtain the expression of the propagator as in \cite{BMW, Ford}.\\

This paper is organized as follows. In Section \ref{sec:2} we review and discuss the properties of the micro-localized spectral measure associated to the operator $\Delta_g$.
In Section \ref{sec:3}, we prove the dispersive estimates and the $L^2$-estimates for the  Klein-Gordon propagator associated with $\Delta_g$. In Sections \ref{sec:4} and \ref{sec:d-e-in}  we prove homogeneous and inhomogeneous Strichartz estimates for the Klein-Gordon equation when the potential term vanishes. In Section \ref{sec:6} we provide a local smoothing estimate for the Klein-Gordon equation, and in Section \ref{sec:7} we eventually prove Theorem \ref{thm:Strichartz}.

\section{The spectral measure and Littlewood-Paley estimates}\label{sec:2}

In this section, we recall the spectral measure of the operator $\Delta_g$ and the Littlewood-Paley square function estimates proved by the last author and Zheng in \cite{ZZ1}.
The properties of the micro-localized spectral measure capture the decay and the oscillatory behavior. These are key tools for showing the micro-localized dispersive estimates.

\begin{proposition}[Proposition 3.1 \cite{ZZ1}]
\label{prop:localized spectral measure} Let $(X,g)$ be metric cone manifold with distance $d(\cdot, \cdot)$ and let $\Delta_g$ be the positive Laplace-Beltrami operator on $X$. Then there exists a $\lambda$-dependent  operator partition of unity on
$L^2(X)$
$$
\mathrm{Id}=\sum_{k=0}^{N}Q_k(\lambda),
$$
with $N$ independent of $\lambda$,
such that for each $1 \leq k \leq N$, we can write \\$(Q_k(\lambda)dE_{\sqrt{\Delta_g}}(\lambda)Q_k^*(\lambda))(z,z')$ as
\begin{equation}\label{beanQ}\begin{gathered}
\lambda^{n-1} \Big(  \sum_{\pm} e^{\pm
i\lambda d(z,z')}a_\pm(\lambda,z,z') +  b(\lambda, z, z') \Big),
\end{gathered}\end{equation}
and for either $k=0$ or $k'=0$ with $0\leq k,k'\leq N$
\begin{equation}\label{beanQ'}\begin{gathered}
(Q_k(\lambda)dE_{\sqrt{\Delta_g}}(\lambda)Q_{k'}^*(\lambda))(z,z')=\lambda^{n-1} c(\lambda, z, z'),
\end{gathered}\end{equation}
with estimates
\begin{equation}\label{bean}\begin{gathered}
\big|\partial_\lambda^\alpha a_\pm(\lambda,z,z') \big|\leq C_\alpha
\lambda^{-\alpha}(1+\lambda d(z,z'))^{-\frac{n-1}2},
\end{gathered}\end{equation}
\begin{equation}\label{beans}\begin{gathered}
\big| \partial_\lambda^\alpha b(\lambda,z,z') \big|\leq C_{\alpha, X}
\lambda^{-\alpha}(1+\lambda d(z,z'))^{-K} \text{ for any } K>0,
\end{gathered}\end{equation}
and
\begin{equation}\label{beanc}\begin{gathered}
\big| \partial_\lambda^\alpha c(\lambda,z,z') \big|\leq C_{\alpha, X}
\lambda^{-\alpha}.
\end{gathered}\end{equation}
\end{proposition}

This proposition is enough to show the validity of Strichartz estimates \eqref{stri} and \eqref{V-stri}, but not to treat the endpoint estimates \eqref{V-stri'}: more precisely, in this latter case we need a double endpoint inhomogeneous Strichartz estimate to deal with the potential term.
To this end, we classify the  $Q_k$ with $k\geq1$ in order to obtain more informations about the micro-localized spectral measure (for example, the sign of phase function). The same idea was used in \cite{GH, HZ,ZZ1}.
For the sake of  convenience, we restate these properties.  \vspace{0.2cm}

Recall that the $Q_k$ with $k\geq1$ (constructed in  \cite[Proposition 3.1]{ZZ1})  are micro-localized away from the cone tip.
By using \cite[Lemma 8.2]{HZ} (see also \cite[Lemmas 5.3 and 5.4]{GH}), we can divide $(k,k')$, $1 \leq k,k' \leq N$ into three classes
\begin{equation}\label{cl-N}
\{ 1, \dots, N \}^2 = J_{near} \cup J_{not-out} \cup J_{not-inc},
\end{equation}
so that
\begin{itemize}
\item
if $(k,k') \in J_{near}$, then $Q_k(\lambda) dE_{\sqrt{\Delta_g}} Q_{k'}(\lambda)^*$ can be written as in \eqref{beanQ}, i.e.
\begin{equation*}\begin{gathered}
\lambda^{n-1} \Big(  \sum_{\pm} e^{\pm
i\lambda d(z,z')}a_\pm(\lambda,z,z') +  b(\lambda, z, z') \Big),
\end{gathered}\end{equation*}
where $a_\pm$ and $b$ satisfy \eqref{beans} and \eqref{beanc}, respectively. That is, we have an estimate analogous to the one in Proposition 
\ref{prop:localized spectral measure}.

\item if $(k,k') \in J_{non-inc}$, then $Q_k(\lambda)$ is not incoming-related to $Q_{k'}(\lambda)$ in the sense that no point in the operator wavefront set (microlocal support)
of $Q_k(\lambda)$ is related to a point in the operator wavefront
set of $Q_{k'}(\lambda)$ by backward bicharacteristic flow;

\item if $(k,k') \in J_{non-out}$, then $Q_k(\lambda)$ is not outgoing-related to $Q_{k'}(\lambda)$ in the sense that no point in the operator wavefront set  of
$Q_k(\lambda)$ is related to a point in the operator wavefront set
of $Q_{k'}(\lambda)$ by forward bicharacteristic flow.
\end{itemize}
Therefore, as in \cite[Lemma 8.3, Lemma 8.5]{HZ}, we have the property of the Schwartz kernel of $Q_k(\lambda) dE_{\sqrt{\Delta_g}} Q_{k'}(\lambda)^*$ stated in \cite[Lemma 4.1]{ZZ2}.
The essential key point is that the phase function in the oscillatory expression of the Schwartz kernel of  the above microlocalized spectral measure has
an unchanged sign when $(k,k') \in J_{non-inc}$ or $(k,k') \in J_{non-out}$. More precisely, there exists a small constant $c>0$ such that the phase function $\Phi\leq -c$ if $(k,k') \in J_{non-out}$
and $\Phi\geq c$ when $(k,k') \in J_{non-inc}$. Let us illustrate the idea.   If $Q_k$
is not outgoing-related to $Q_{k'}$, since the other terms in \cite[Lemma 4.1]{ZZ2} follow the same idea, we only consider 
\begin{equation}\label{QkEQk'}
\begin{split}
Q_k(\lambda) dE_{\sqrt{\Delta_g}}(\lambda)Q_{k'}(\lambda)^*=\int_{\R^\ell} e^{i\lambda\Phi(z,z',v)}\lambda^{n-1+\frac \ell2}a(\lambda,z,z',v) dv
\end{split}
\end{equation}
where $\Phi(z,z',v)\leq -c<0$ and $|(\lambda\partial_\lambda)^\alpha a|\leq C_\alpha $.
Here the parameter $0\leq \ell\leq n-1$ is connected to degenerate rank of Jacobi field along geodesic connecting the points $z,z'$ (which corresponds to the degenerate rank of the projection from the phase space to the base). Following the previous result in \cite{GHS, HZ} and references therein, one finds that if $\ell=0$ (which occurs if there is no conjugate points in the geodesic), then the expression (without integration) is similar to \eqref{beanQ} in which the conjugate points are separated. If $\ell>0$, the microlocalized propagator is
\begin{equation}\label{mic-propagator}
\begin{split}
\int_0^\infty e^{i(t-\tau)\sqrt{1+\lambda^2}}\int_{\R^\ell} e^{i\lambda\Phi(z,z',v)}\lambda^{n-1+\frac \ell2}a(\lambda,z,z',v) dv d\lambda,
\end{split}
\end{equation} 
and this brings a difficulty to obtain the 
dispersive estimates at high frequency. 
If we restrict to $\tau>t$, then the derivative of the phase function has a positive lower bound due to the fact that $\Phi$ and $t-\tau$ have the same sign. Hence, we can overcome these difficulties by integration by parts to obtain a microlocalized dispersive estimates (see Subsection \ref{sec:mic-dis}).

\vspace{0.2cm}

For our purpose,  we need the Littlewood-Paley square function estimates associated to the operator $\Delta_g$.  Let $\varphi\in C_c^\infty(\mathbb{R}\setminus\{0\})$ take values in
$[0,1]$ and be compactly supported in $[1/2,2]$ such that  
\begin{equation}\label{dp}
1=\sum_{j\in\Z}\varphi(2^{-j}\lambda), \quad \varphi_0(\lambda):=\sum_{j\leq0} \varphi(2^{-j}\lambda),\quad\lambda>0.
\end{equation}

\begin{proposition}[\cite{ZZ2} Proposition 2.6]\label{prop:square}  Let $(X,g)$ be a metric cone of dimension $n\geq3$ and let $\Delta_g$ be the
positive Laplace-Beltrami operator on $X$. Then for $1<p<\infty$,
there exist constants $c_p$ and $C_p$ depending on $p$ such that
\begin{equation}\label{square}
c_p\|f\|_{L^p(X)}\leq
\big\|\big(\sum_{j\in\Z}|\varphi(2^{-j}\sqrt{\Delta_g})f|^2\big)^{\frac12}\big\|_{L^p(X)}\leq
C_p\|f\|_{L^p(X)}
\end{equation}
and
\begin{equation}\label{squareL}
c_p\|f\|_{L^{p,2}(X)}\leq
\big\|\big(\sum_{j\in\Z}|\varphi(2^{-j}\sqrt{\Delta_g})f|^2\big)^{\frac12}\big\|_{L^{p,2}(X)}\leq
C_p\|f\|_{L^{p,2}(X)}
\end{equation}
where with $L^{p,2}(X)$ we are denoting the standard Lorentz spaces on $X$.
\end{proposition}

\section{Microlocalized dispersive estimates and $L^2$-estimates}\label{sec:3}
In this section, we prove  dispersive  and  $L^2$-estimates for the  Klein-Gordon propagator associated to $\Delta_g$.
Since at this stage the potential term $V=V_0(y)r^{-2}$ is not taken into account, the main difficulties arise from the existence of conjugate points: we will make use of  microlocal techniques to separate them and establish  dispersive estimates.

\subsection{Microlocalized propagator and $L^2$-estimates}
We first define the microlocalized Klein-Gordon propagator. Denote $U(t)=e^{it\sqrt{1+\Delta_g}}$. For any $\sigma\geq0$, we define 
\begin{equation*}
\leftidx{^{\sigma}}U(t)=e^{it\sqrt{1+\Delta_g}}(1+\Delta_g)^{-\frac\sigma 2}.
\end{equation*}
In the following pages we will especially focus on the cases $\sigma=0$, $\sigma=1/2$ and $\sigma=1$.
Let $\varphi\in C_c^\infty([1/2, 2])$ as in \eqref{dp},
and define
\begin{equation}\label{U-dyn}
\begin{split}
\leftidx{^{\sigma}}U_j(t) &= \int_0^\infty e^{it\sqrt{1+\lambda^2}}
\varphi(2^{-j}\lambda) \langle \lambda\rangle^{-\sigma}dE_{\sqrt{\Delta_g}}(\lambda),\quad j\in \Z,
\end{split}
\end{equation}
where we are using the standard notation for the  bracket $\langle \lambda\rangle=(1+\lambda^2)^{1/2}$.
Using the $Q_k(\lambda)$ defined in Proposition \ref{prop:localized spectral measure}, we further define
\begin{equation}\label{U-mic}
\begin{gathered}
\leftidx{^{\sigma}}U_{j,k}(t) = \int_0^\infty e^{it\sqrt{1+\lambda^2}} \varphi(2^{-j}\lambda)\langle \lambda\rangle^{-\sigma}Q_{k}(\lambda)
dE_{\sqrt{\Delta_g}}(\lambda),~j\in\Z,\,  0\leq k\leq N,
\end{gathered}
\end{equation}
and
\begin{equation}\label{U-l-mic}
\begin{gathered}
\leftidx{^{\sigma}}U_{\mathrm{low},k}(t)=\sum_{j\leq0}\leftidx{^{\sigma}}U_{j,k}(t)  = \int_0^\infty e^{it\sqrt{1+\lambda^2}} \varphi_0(\lambda)\langle \lambda\rangle^{-\sigma}Q_{k}(\lambda)
dE_{\sqrt{\Delta_g}}(\lambda),
\end{gathered}
\end{equation}
where $\varphi_0(\lambda)=\sum_{j\leq0}\varphi(2^{-j}\lambda)$.
The above definitions of the operator are well-posed: indeed, we have
\begin{proposition}[$L^2$-estimates]\label{energy} For all $0\leq k\leq N$, there exists a constant $C$ independent of $t, j, k$ such that
\begin{equation}\label{L2-est-l}
\begin{split}
\|\leftidx{^{\sigma}}U_{j,k}(t)\|_{L^2\rightarrow L^2}\leq C,\quad j\leq 0,
\end{split}
\end{equation}
and 
\begin{equation}\label{L2-est-h}
\begin{split}
\|\leftidx{^{\sigma}}U_{j,k}(t)\|_{L^2\rightarrow L^2}\leq C2^{-j\cdot\sigma},\quad j\geq 0.
\end{split}
\end{equation}
\end{proposition}
\begin{proof} The proof essentially follows the argument in \cite{HZ} in which Hassell and the last author considered the case of asymptotically conic manifolds. 
We  outline the proof for convenience. 

We first show that the above definition of the operator is well-posed.  To this end, it suffices
to show that the integrals in the definitions above are well-defined over any compact
dyadic interval in $(0, +\infty)$. Let $A(\lambda) = e^{it\sqrt{1+\lambda^2}}
\varphi(2^{-j}\lambda)\langle \lambda\rangle^{-\sigma}Q_{k}(\lambda)$. Then $A(\lambda)$ is compactly supported in $[a,b]$ with $a=2^{j-1}$ and $b=2^{j+1}$
and is $\mathcal{C}^1$ in $\lambda\in (0,+\infty)$. Integrating by parts, the
integral 
$$
\int_a^b A(\lambda) dE_{\sqrt{\Delta_g}}(\lambda)
$$
becomes
\begin{equation}\label{mean}
E_{\sqrt{\Delta_g}}(b) A(b) - E_{\sqrt{\Delta_g}}(a)
A(a) - \int_a^b \frac{d}{d\lambda} A(\lambda)
E_{\sqrt{\Delta_g}}(\lambda) \, d\lambda.
\end{equation}
From the construction of the pseudo-differential operator $Q_k(\lambda)$ in \cite{ZZ1} and \cite[Corollary 3.3]{HZ}, we can show that  $Q_k(\lambda)$ and each operator $\lambda
\partial_\lambda Q_k(\lambda)$ is bounded on $L^2(X)$
uniformly in $\lambda$. This implies that the integrals are well-defined over any dyadic compact
interval in $(0, +\infty)$, hence the operators $\leftidx{^{\sigma}}U_{j,k}(t)$ are
well-defined. 

Next we show that these operators are bounded on $L^2$. We have by  \cite[Lemma 5.3]{HZ},
\begin{equation}\begin{gathered}
\leftidx{^{\sigma}}U_{j,k}(t) \leftidx{^{\sigma}}U_{j,k}(t)^* = \int   \varphi\big(
\frac{\lambda}{2^j} \big) \varphi\big( \frac{\lambda}{2^j} \big)\langle \lambda\rangle^{-2\sigma}
Q_k(\lambda) dE_{\sqrt{\Delta_g}}(\lambda) Q_k(\lambda)^* \\
= -\int \frac{d}{d\lambda} \Big(  \varphi\big(
\frac{\lambda}{2^j} \big) \varphi\big( \frac{\lambda}{2^j} \big)
Q_k(\lambda) \langle \lambda\rangle^{-2\sigma}\Big) E_{\sqrt{\Delta_g}}(\lambda) Q_k(\lambda)^*d\lambda  \\
- \int  \varphi\big( \frac{\lambda}{2^j} \big)
\varphi\big( \frac{\lambda}{2^j} \big)\langle \lambda\rangle^{-2\sigma} Q_k(\lambda)
E_{\sqrt{\Delta_g}}(\lambda) \frac{d}{d\lambda}
Q_k(\lambda)^*d\lambda.
\end{gathered}\label{Uijk}\end{equation}
We first note that this is independent of $t$. We also recall that $Q_k(\lambda)$ and $\lambda
\partial_\lambda Q_k(\lambda)$ are bounded on $L^2(X)$
uniformly in $\lambda$. On the other hand,  the
integrand is a bounded operator on $L^2$, with an operator bound of
the form $C\lambda^{-1}\langle \lambda\rangle^{-2\sigma}$ where $C$ is uniform in $\lambda$. and, by the support properties of $\varphi$, we have that the
$L^2$-operator norm of the integral is therefore uniformly bounded by $(1+2^{2j})^{-\sigma}$, as we are integrating over a dyadic interval in $\lambda$.  This proves Proposition \ref{energy}.
\end{proof}

As a direct consequence, we have the following result:

\begin{proposition}[$L^2$-estimates for low energy]\label{energy'} Let $\leftidx{^{\sigma}}U_{\mathrm{low},k}(t)$ be defined in \eqref{U-l-mic}.
Then there exists a constant $C$ independent of $t, z, z'$ such that
$\|\leftidx{^{\sigma}}U_{\mathrm{low},k}(t)\|_{L^2\rightarrow L^2}\leq C$ for all $0\leq k\leq N$ and $\sigma\geq0$.
\end{proposition}

\begin{proof}
Let $f\in L^2$, by using the Littlewood-Paley theory in Proposition \ref{prop:square} , we have 
\begin{equation*}
\begin{split}
\|\leftidx{^{\sigma}}U_{\mathrm{low},k}(t)f\|^2_{L^2}&=\sum_{\ell\in\Z}\|\varphi(2^{-\ell}\sqrt{\Delta_g})\,\leftidx{^{\sigma}}U_{\mathrm{low},k}(t)f\|^2_{L^2}\\
&=\sum_{\ell\in\Z}\|\leftidx{^{\sigma}}U_{\mathrm{low},k}(t)\varphi(2^{-\ell}\sqrt{\Delta_g})\tilde{\varphi}(2^{-\ell}\sqrt{\Delta_g})f\|^2_{L^2}\\
&\leq \sum_{\ell\in\Z}\|\sum_{j\leq0} \leftidx{^{\sigma}}{\tilde U}_{j,\ell,k}(t) \tilde{\varphi}(2^{-\ell}\sqrt{\Delta_g})f\|^2_{L^2}
\leq \|f\|^2_{L^2}.
\end{split}
\end{equation*}
where we choose $\tilde{\varphi}\in C_c^\infty([1/4, 4])$ such that $\varphi=\tilde{\varphi}\varphi$ and
\begin{equation*}
\begin{gathered}
\leftidx{^{\sigma}}{\tilde U}_{j,\ell,k}(t) = \int_0^\infty e^{it\sqrt{1+\lambda^2}} \varphi(2^{-j}\lambda) \varphi(2^{-\ell}\lambda)\langle \lambda\rangle^{-\sigma}Q_{k}(\lambda)
dE_{\sqrt{\Delta_g}}(\lambda)
\end{gathered}
\end{equation*}
which vanishes when $|j-\ell|\geq 5$ and satisfies \eqref{L2-est-l}. Therefore, by \eqref{L2-est-l}, we obtain
\begin{equation*}
\begin{split}
\|\leftidx{^{\sigma}}U_{\mathrm{low},k}(t)f\|^2_{L^2}
&\leq \sum_{\ell\in\Z}\left(\sum_{|j-\ell|\leq 5}\| \leftidx{^{\sigma}}{\tilde U}_{j,\ell,k}(t) \tilde{\varphi}(2^{-\ell}\sqrt{\Delta_g})f\|_{L^2}\right)^2\\&\leq \sum_{\ell\in\Z}\|\tilde{\varphi}(2^{-\ell}\sqrt{\Delta_g})f\|^2_{L^2}
\leq \|f\|^2_{L^2}.
\end{split}
\end{equation*}
This  completes the proof of Proposition \ref{energy'}.
\end{proof}

\subsection{Microlocalized dispersive estimates} \label{sec:mic-dis} In this subsection, we prove the dispersive estimates for the microlocalized propagators $\leftidx{^{\sigma}}U_{\bullet,k}(t) \leftidx{^{\sigma}}U_{\bullet,k'}(\tau)^*$ with suitable pairs $(k,k')$. 
The restriction on the pairs is necessary because of the existence of conjugate points.
To this end, we divide  the proof of the dispersive estimates into two cases:  low frequencies and high frequencies. We stress the fact that this approach seems to be quite natural as indeed the dispersion of the Klein-Gordon equation resembles  Schr\"odinger  for low frequencies and the wave equation for high frequencies.

\subsubsection{Low frequency estimates} First, we prove the microlocalized dispersive estimates for the low energy part.

\begin{proposition}[Low energy estimates]\label{prop:lDispersive} Let  $\leftidx{^{\sigma}}U_{\mathrm{low},k}(t) $ be defined as in  \eqref{U-l-mic}.
Then there exists a constant $C$ independent of $t, \tau, z, z'$, for all
$0\leq k,k'\leq N$ and $\sigma\geq0$, such that the dispersive estimate
\begin{equation} \label{l-}
\big\|\leftidx{^{\sigma}}U_{\mathrm{low},k}(t) (\leftidx{^{\sigma}}U_{\mathrm{low},k'}(\tau) )^*\big\|_{L^1\rightarrow L^\infty}\leq 
C (1+|t-\tau|)^{-\frac{n}2}; \end{equation}
holds if one of the following conditions holds:

\begin{itemize}
\item \label{l-near} $(k,k') \in J_{near}$ or $(k,k')=(0,k'), (k,0)$ and $t \neq \tau$;

\item \label{l-nonout} $(k,k')\in J_{non-out}$ and $t<\tau$;

\item \label{l-nonin} $(k,k')\in J_{non-inc}$ and $t>\tau$.

\end{itemize}

\end{proposition}

\begin{remark}
As expected, the decay rate in the above dispersive estimates \eqref{l-} is $O((1+|t-\tau|)^{-n/2})$ which is same as the Schr\"odinger decay rate.
The result is independent of the index $\sigma$ since we cannot gain anything from the factor $\langle \lambda\rangle^{-2\sigma}$ at low energy.

\end{remark}

\begin{remark}
The last two cases are only needed to prove the endpoint Strichartz estimates \eqref{V-stri'}.
\end{remark}

\vspace{0.2cm}

\begin{proof}
In the case of the first condition we use a stationary phase argument together with Proposition \ref{prop:localized spectral measure} to prove \eqref{l-}. The proof when each of the last two conditions are satisfied relies on \cite[Lemma 4.1]{ZZ2}  instead.

We first consider the  second and third cases, where  \cite[Lemma 4.1]{ZZ2} is used. Furthermore, we only prove \eqref{l-} in the second case, since the third case follows from an analogous argument.
Recalling that $\varphi_0$ is given in \eqref{dp}, we set $\phi_0=\varphi_0^2$.
Under the assumption that $Q_k$ is not outgoing-related to $Q_{k'}$ and $\tau>t$,  we need to show 
\begin{equation}\label{l-nonout'}
\begin{split}
&\Big|\int_0^\infty e^{i(t-\tau)\sqrt{1+\lambda^2}}\chi_0(\lambda) \big(Q_k(\lambda)
dE_{\sqrt{\Delta_g}}(\lambda)Q^*_{k'}(\lambda)\big)(z,z')\Big|\\&\leq C (1+|t-\tau|)^{-n/2},
\end{split}
\end{equation}
where $\chi_0(\lambda)=\langle \lambda\rangle^{-2\sigma}\phi_0(\lambda)$ satisfies the same property of $\varphi_0$ since $0\leq \lambda\leq 2$.\vspace{0.2cm}

If $t-\tau>-1$, then one has $|t-\tau|\leq 1$ since $t<\tau$. Since $\chi_0$ is compactly supported in $[0,2]$, the estimate \eqref{l-nonout'} follows from the uniform boundedness of 
\cite[Lemma 4.1, (4.4)-(4.6)]{ZZ2}.
Thus it suffices to consider the case $t-\tau \leq -1$. As  in \eqref{dp}, we choose $\tilde{\varphi}\in C_c^\infty([\frac12,2])$ to be  such that
$\sum_m\tilde{\varphi}(2^{-m}(\tau-t)\lambda)=1$, and define
$$\tilde{\varphi}_0((\tau-t)\lambda)=\sum_{m\leq0}\tilde{\varphi}_m((\tau-t)\lambda), \quad \tilde{\varphi}_m((\tau-t)\lambda)=\tilde{\varphi}(2^{-m}(\tau-t)\lambda).$$ 
Then it suffices to show
\begin{equation}\label{l-nonout'-l}
\begin{split}
&\Big|\int_0^\infty e^{i(t-\tau)\sqrt{1+\lambda^2}}\chi_0(\lambda)\tilde{\varphi}_0((\tau-t)\lambda)\big(Q_k(\lambda)
dE_{\sqrt{\Delta_g}}(\lambda)Q^*_{k'}(\lambda)\big)(z,z')\Big|\\&\leq C (1+|t-\tau|)^{-n/2},
\end{split}
\end{equation}
and 
\begin{equation}\label{l-nonout'-h}
\begin{split}
&\sum_{m\geq1}\Big|\int_0^\infty e^{i(t-\tau)\sqrt{1+\lambda^2}}\chi_0(\lambda)\tilde{\varphi}_m((\tau-t)\lambda) \big(Q_k(\lambda)
dE_{\sqrt{\Delta_g}}(\lambda)Q^*_{k'}(\lambda)\big)(z,z')\Big|\\&\leq C (1+|t-\tau|)^{-n/2}.
\end{split}
\end{equation}
By using the expression of spectral measure \cite[Lemma 4.1, (4.4)-(4.6)]{ZZ2},
for any $0\leq \ell\leq n-1$ (we are replacing $k$ in \cite{ZZ2} by $\ell$ to avoid confusion), we obtain 
\begin{equation}
\begin{split}
\text{LHS of}\, \eqref{l-nonout'-l}\lesssim \int_0^\infty \lambda^{n-1+\frac \ell 2} \chi_0(\lambda)\tilde{\varphi}_0((\tau-t)\lambda) d\lambda\leq C |t-\tau|^{-n}.
\end{split}
\end{equation}
which implies \eqref{l-nonout'-l} since $|t-\tau|>1$.\vspace{0.2cm}

Next we prove \eqref{l-nonout'-h}. For each $\ell \geq 1$, we only consider  \cite[Lemma 4.1, (4.6)]{ZZ2} for simplicity, which is represented here in \eqref{QkEQk'} and \eqref{mic-propagator},  since the other two cases follow from a similar argument.
Set $\tilde{\lambda} = (\tau-t)\lambda$.  By using scaling and \eqref{QkEQk'}, we obtain
\begin{equation}\label{l-nonout'-h'}
\begin{split}
\text{LHS of}\, \eqref{l-nonout'-h}&= (\tau-t)^{-n - \frac{\ell}{2}}\int_0^\infty \int_{\R^{\ell}}
e^{i\big(-\sqrt{(\tau-t)^2+\tilde{\lambda}^2}+\frac{\tilde{\lambda}\Phi(z,z',v)}{\tau-t}\big)}\tilde{\lambda}^{n-1+\frac{\ell}2}
\\& \qquad\qquad\qquad\qquad \times \chi_0(\frac{\tilde{\lambda}}{\tau-t})\tilde{\varphi}_m(\tilde{\lambda}) a(\frac{\tilde{\lambda}}{\tau-t},y,y',\sigma,v)  \,
dv \, d\tilde{\lambda},\end{split} 
\end{equation} 
where $\Phi(z,z',v)\leq -c<0$ and $|(\lambda\partial_\lambda)^\alpha a|\leq C_\alpha $.
Define the operator
$$
L = \frac{i}{ \frac{-\Phi}{(\tau-t)}+\frac{2\tilde{\lambda}}{\sqrt{(\tau-t)^2+\tilde{\lambda}^2}}}\frac{\partial}{\partial \tilde{\lambda}}.
$$
Then by using $\tau-t\geq 1$ and $\Phi\leq -c$, we can use the induction argument to prove the facts $|(\tilde{\lambda} \partial_{\tilde{\lambda}} )^N (\chi_0 \tilde{\varphi}_m a)| \leq
C_N$ and
\begin{equation*}
\left|\big(\frac{\partial}{\partial\tilde{\lambda}}\big)^N\left( \Big(\frac{-\Phi}{(\tau-t)}+\frac{2\tilde{\lambda}}{\sqrt{(\tau-t)^2+\tilde{\lambda}^2}}\Big)^{-1}\right)\right|\lesssim \tilde{\lambda}^{-N}.
\end{equation*}
 Thus we gain a factor $\tilde{\lambda}^{-1}
\sim 2^{-m}$ via integration by parts each time. Hence for $\tau-t\geq1$ and $0\leq \ell\leq n-1$, we have
\begin{equation*}\begin{split}
\text{LHS of}\, \eqref{l-nonout'-h}\lesssim (\tau-t)^{-n - \frac{\ell}{2}}\int_{\tilde{\lambda}\sim 2^m}
\tilde{\lambda}^{n-1+\frac{\ell}2-N}\, d\tilde{\lambda}
\lesssim (\tau-t)^{-n} 2^{-m(N-n - \frac{\ell}2)}.
\end{split} 
\end{equation*}
Choosing $N$ large enough, we obtain \eqref{l-nonout'-h} by summing over
$m\geq1$, and thus we obtain \eqref{l-} when the second condition holds.\vspace{0.2cm}

We now prove \eqref{l-} when the first condition holds. Due to \eqref{cl-N}, the spectral measure
$Q_k(\lambda) dE_{\sqrt{\Delta_g}} Q_{k'}(\lambda)^*$ satisfies the conclusions of Proposition~\ref{prop:localized spectral measure} when either
$(k,k') \in J_{near}$, $k=0$ or $k'=0$.
Estimate \eqref{l-} then follows from Lemma \ref{dispersive-l}.
\end{proof}

 \begin{lemma}[Microlocalized dispersive estimates for low frequencies]\label{dispersive-l}
Assume  $(k,k')\in \{ 0, 1, \dots, N \}^2$ such that 
$Q_k(\lambda) dE_{\sqrt{\Delta_g}} Q_{k'}(\lambda)^*$ has a bound of the form appearing in \eqref{beanQ}.
Then there exists a constant $C$ independent of points $z,z'\in X$.
such that
\begin{equation}\label{disper-near}
\Big|\int_0^\infty e^{it\sqrt{1+\lambda^2}} \chi_0(\lambda)\big(Q_k(\lambda)
dE_{\sqrt{\Delta_g}}(\lambda)Q_{k'}^*(\lambda)\big)(z,z')\Big|\leq C (1+|t|)^{-\frac{n}2}
\end{equation}
where $\chi_0(\lambda)=\varphi^2_0(\lambda)\langle\lambda\rangle^{-2\sigma}$ and $\varphi_0\in C_c^\infty([0,2])$ is as in \eqref{dp}.
\end{lemma}

\begin{proof} 
By definition, we have that $\chi_0\in C_c^\infty([0, 2])$ and satisfies $|(\lambda\partial_\lambda)^{\ell}\chi_0(\lambda)|\leq C_\ell$ for all $\ell\geq0$. 
From Proposition~\ref{prop:localized spectral measure},
we need to consider 
\begin{equation}\label{ab}
\begin{split}
&\int_0^\infty e^{it\sqrt{1+\lambda^2}} \chi_0(\lambda)\big(Q_k(\lambda)
dE_{\sqrt{\Delta_g}}(\lambda)Q_{k'}^*(\lambda)\big)(z,z')
\\
&=\sum_\pm \int_0^\infty e^{it\sqrt{1+\lambda^2}}e^{\pm
i\lambda d(z,z')}\lambda^{n-1}\chi_0(\lambda)a_\pm(\lambda,z,z')d\lambda\\ &\quad+\int_0^\infty
e^{it\sqrt{1+\lambda^2}}\lambda^{n-1}\chi_0(\lambda)b(\lambda,z,z')d\lambda  
\end{split}
\end{equation}
 for either $(k,k')=(k,0)$ or $(k,k')=(0,k')$,
\begin{equation}\label{c}
\begin{split}
&\int_0^\infty e^{it\sqrt{1+\lambda^2}} \chi_0(\lambda)\big(Q_k(\lambda)
dE_{\sqrt{\Delta_g}}(\lambda)Q_{k'}^*(\lambda)\big)(z,z')
\\
&=\int_0^\infty
e^{it\sqrt{1+\lambda^2}}\lambda^{n-1}\chi_0(\lambda)c(\lambda,z,z')d\lambda  .
\end{split}
\end{equation}
where $a_\pm,b$ and $c$ are from Proposition ~\ref{prop:localized spectral measure}.

If $|t|\lesssim 1$, then \eqref{disper-near} directly follows from the compact support of $\chi_0$ and the boundednesses of $a_\pm,b$ and $c$ in Proposition ~\ref{prop:localized spectral measure}. Therefore from now on, by symmetry in time, it will be enough to consider 
the case $t\gg1$. 

Since \eqref{c} can be treated by following the argument of \eqref{ab} with the term $b$, we only consider \eqref{ab}. First we estimate \eqref{ab} when $t\gg1$. Set $r=d(z,z')$
and $\tilde{r}=r/\sqrt{t}$.  By scaling, it is enough to estimate
\begin{equation}\label{4.4?}
\begin{split}
&\int_0^\infty e^{it\sqrt{1+\lambda^2}} \chi_0(\lambda)\big(Q_k(\lambda)
dE_{\sqrt{\Delta_g}}(\lambda)Q_{k'}^*(\lambda)\big)(z,z')\\
&= t^{-\frac
n2}\sum_{\pm} \int_0^{\infty} e^{i\sqrt{t^2+t\lambda^2}}e^{\pm
i\tilde{r}\lambda}\lambda^{n-1}\chi_0(t^{-1/2}\lambda)a_\pm(t^{-1/2}\lambda,z,z')d\lambda
\\&\quad+\int_0^\infty e^{it\sqrt{1+\lambda^2}}\lambda^{n-1}\chi_0(\lambda)b(\lambda,z,z')d\lambda,
\end{split}
\end{equation}
where $a_\pm$ satisfies \eqref{bean}, hence
\begin{equation}\label{beans0-l}
\Big|\partial_\lambda^\alpha \big(a_\pm(t^{-1/2}\lambda,z,z')\big)
\Big|\leq C_\alpha \lambda^{-\alpha}(1+\lambda
\tilde{r})^{-\frac{n-1}2}.
\end{equation}
First, we estimate the second term on the RHS of   \eqref{4.4?}.
Recalling \eqref{beans}, we obtain
\begin{equation}\label{4.2}
\begin{split}
\Big|\big(\frac{\partial}{\partial\lambda}\big)^{N}b(\lambda,z,z')\Big|\leq
C_N\lambda^{-N}\quad \forall N\in\mathbb{N}.
\end{split}
\end{equation}
Let $\delta$ be a small constant to be chosen later and recall $\varphi$ and $\varphi_0$ defined in \eqref{dp}. Then
\begin{equation*}
\Big|\int_0^\infty e^{it\sqrt{1+\lambda^2}} \lambda^{n-1}b(\lambda,z,z')
\chi_0(\lambda)\varphi_0(\frac{\lambda}{\delta})d\lambda\Big|\leq
C\int_0^{2\delta}\lambda^{n-1}d\lambda\leq C\delta^n.
\end{equation*}
By \eqref{4.2},  we use  integration by parts $N$ times to obtain 
\begin{equation*}
\begin{split}
&\Big|\int_0^\infty e^{it\sqrt{1+\lambda^2}}
\chi_0(\lambda)\sum_{m\geq 1}\varphi\big(\frac{\lambda}{2^m\delta}\big)\lambda^{n-1}b(\lambda,z,z')
d\lambda\Big|\\
&\leq \sum_{m\geq 1}\Big|\int_0^\infty
\big(\frac{\sqrt{1+\lambda^2}}{i\lambda
t}\frac\partial{\partial\lambda}\big)^{N}\big(e^{it\sqrt{1+\lambda^2}}\big)
\chi_0(\lambda)
\varphi\big(\frac{\lambda}{2^m\delta}\big)\lambda^{n-1}b(\lambda,z,z')
d\lambda\Big|\\& \leq
C_N|t|^{-N}\sum_{m\geq1}\int_{2^{m-1}\delta}^{2^{m+1}\delta}\lambda^{n-1-2N}d\lambda\leq
C_N|t|^{-N}\delta^{n-2N},\quad N>\frac n2.
\end{split}
\end{equation*}
Choosing $\delta=|t|^{-\frac12}$, we have thus proved
\begin{equation}\label{4.3}
\begin{split}
&\Big|\int_0^\infty e^{it\sqrt{1+\lambda^2}} \chi_0(\lambda)\lambda^{n-1} b(\lambda,z,z')
d\lambda\Big|\leq C_N|t|^{-\frac n2}.
\end{split}
\end{equation}

 Next we consider the first term on the RHS of \eqref{4.4?}. As above, by using $\varphi$ and $\varphi_0$ as defined in \eqref{dp}, we split it into two parts. It suffices to prove that there exists a constant $C$
independent of $\tilde{r}$ and $t$ such that
\begin{equation*}
\begin{split}
I^\pm:=&\Big|\int_0^{\infty} e^{i\sqrt{t^2+t\lambda^2}}e^{\pm
i\tilde{r}\lambda}\lambda^{n-1}\chi_0(t^{-1/2}\lambda)a_\pm(t^{-1/2}\lambda,z,z')\varphi_0(\lambda)d\lambda\Big|\leq
C,\\II^\pm:=& \Big|\sum_{m\geq1}\int_0^{\infty} e^{i\sqrt{t^2+t\lambda^2}}e^{\pm
i\tilde{r}\lambda}\lambda^{n-1} \chi_0(t^{-1/2}\lambda) a_\pm(t^{-1/2}\lambda,z,z')\varphi(\frac{\lambda}{2^m})d\lambda\Big|\leq
C.
\end{split}
\end{equation*}
The estimate for $I^\pm$ is relatively straightforward by using the fact that the support of $\varphi_0$ is restricted to $\lambda\leq 2$ and by using \eqref{beans0-l}.
For $II^{+}$, we use $N$-times integration by parts to gain $\lambda^{-2N}$. 

Indeed, we first note
$$
e^{i\sqrt{t^2+t\lambda^2} + i \tilde{r} \lambda}= (L^+)^N (e^{i\sqrt{t^2+t\lambda^2} + i \tilde{r} \lambda}), \quad L^+ = \frac1i\big(\frac{t\lambda}{\sqrt{t^2+t\lambda^2}} + \tilde{r}\big)^{-1} \frac{\partial}{\partial \lambda}.
$$
On the support of $\chi_0$, it gives $0<\lambda<2\sqrt{t}$. Then  for $\ell\geq0$ and $t\gg1$, the induction argument follows
\begin{equation}\label{I-bp4}
\begin{split}
\partial^\ell_\lambda\Big[\big(\frac{t\lambda}{\sqrt{t^2+t\lambda^2}}+\tilde{r}\big)^{-1}\Big]\leq C_\ell \lambda^{-1-\ell}.
\end{split}
\end{equation}
By using \eqref{beans0-l} and \eqref{I-bp4}, 
we obtain
$$
II^{+}\lesssim \sum_{m \geq 1} \int_{\lambda \sim 2^m} \lambda^{n-1-2N} \, d\lambda \leq C.
$$
Now we  treat $II^-$. We first write
$II^-=II^-_1+II^-_2$, where (dropping the $-$ superscripts from here on)
\begin{equation*}
\begin{split}
II_1=&\Big|\sum_{m\geq1}\int_0^{\infty} e^{i\sqrt{t^2+t\lambda^2}}e^{-
i\tilde{r}\lambda}\lambda^{n-1}\chi_0(t^{-1/2}\lambda) a(t^{-1/2}\lambda,z,z')\varphi(\frac{\lambda}{2^m})
\varphi_0(8\tilde{r} \lambda) d\lambda\Big|, ~\\II_2=&\Big|\int_0^{\infty}
e^{i\sqrt{t^2+t\lambda^2}}e^{-
i\tilde{r}\lambda}\lambda^{n-1}\chi_0(t^{-1/2}\lambda)a(t^{-1/2}\lambda,z,z')
\left(1-\varphi_0(\lambda)\right) \big( 1 - \varphi_0(8\tilde{r} \lambda)
\big) d\lambda\Big|.
\end{split}
\end{equation*}
Let $\Phi(\lambda,\tilde{r})=\sqrt{t^2+t\lambda^2}-\tilde{r}\lambda$. We first
consider $II_1$. The integrand in $II_1$ vanishes 
when $\tilde{r}>1/8$ due to the supports of $\varphi$ and $\varphi_0$ (which implies $\lambda \leq (8 \tilde{r})^{-1}$ and $\lambda \geq
1$). Thus $1\leq \lambda<2\sqrt{t}$ and $\tilde{r}\leq1/8$, therefore $|\partial_\lambda\Phi| =
\frac{t\lambda}{\sqrt{t^2+t\lambda^2}} - \tilde{r} \geq \frac1{\sqrt{5}}\lambda- \tilde{r} \geq\frac1{100}\lambda$. As in \eqref{I-bp4}, 
on the support of $\chi_0(\lambda/\sqrt{t})$, for $\ell\geq0$ and $t\gg1$, we also use the induction argument to obtain
\begin{equation}\label{I-bp4'}
\begin{split}
\partial^\ell_\lambda\Big[\big(\frac{t\lambda}{\sqrt{t^2+t\lambda^2}}-\tilde{r}\big)^{-1}\Big]\leq C_\ell \big(\frac{t\lambda}{\sqrt{t^2+t\lambda^2}}-\tilde{r}\big)^{-1}\lambda^{-\ell}\leq C_\ell\lambda^{-1-\ell}.
\end{split}
\end{equation}
 Define the operator
$L=L(\lambda,\tilde{r})=(\frac{t\lambda}{\sqrt{t^2+t\lambda^2}}-\tilde{r} )^{-1}\partial_\lambda$. By using
\eqref{beans0-l} and integration by parts again, we obtain
\begin{equation*}
\begin{split}
II_1\leq&\sum_{m\geq1}\Big|\int_0^{\infty} L^{N}
\big(e^{i(\frac{\lambda}{\sqrt{t^2+t\lambda^2}}-
\tilde{r}\lambda)}\big)\Big[\lambda^{n-1}\chi_0(t^{-1/2}\lambda)a(t^{-1/2}\lambda,z,z')\varphi(\frac{\lambda}{2^m})\varphi_0(8\tilde{r}
\lambda) \Big]d\lambda\Big|
\\\leq &C_N\sum_{m\geq1}\int_{\lambda\sim
2^{m}}\lambda^{n-1-2N}d\lambda\leq C_N.
\end{split}
\end{equation*}
Finally we estimate $II_2$. Based on the size of $\partial_\lambda \Phi$, we make a further decomposition of $II_2$
\begin{equation*}
\begin{split}
II_2\leq &\Big|\int_0^{\infty} e^{i\sqrt{t^2+t\lambda^2}}e^{-
i\tilde{r}\lambda}\lambda^{n-1}\chi_0(t^{-1/2}\lambda) a(t^{-1/2}\lambda,z,z')\\ &\qquad \qquad \big(1-\varphi_0(\lambda)\big)\varphi_0(\frac{t\lambda}{\sqrt{t^2+t\lambda^2}}-\tilde{r}) \big( 1 - \varphi_0(8\tilde{r} \lambda) \big) \, d\lambda\Big|\\
&+\sum_{m\geq1}\Big|\int_0^{\infty}
e^{i\sqrt{t^2+t\lambda^2}}e^{-
i\tilde{r}\lambda}\lambda^{n-1}\chi_0(t^{-1/2}\lambda) a(t^{-1/2}\lambda,z,z')
\\&\qquad \qquad \big(1-\varphi_0(\lambda)\big)\varphi\big(\frac{\frac{t\lambda}{\sqrt{t^2+t\lambda^2}}-\tilde{r}}{2^m}\big)\big( 1 - \varphi_0(8\tilde{r} \lambda) \big) \, d\lambda\Big|\\:=&II_2^1+II_2^2.
\end{split}
\end{equation*}
Due to  the compact support of the second $\varphi_0$ factor in $II_2^1$, one has
\begin{equation}\label{II_2}
|\frac{t\lambda}{\sqrt{t^2+t\lambda^2}}-\tilde{r}|\leq 1.
\end{equation}
If $\tilde{r} \leq 10$, from $\lambda<2\sqrt{t}$ again, we must have $\lambda \leq 100$ otherwise the integrand of $II_2^1$ vanishes.
Then we see that $II_2^1$ is uniformly bounded. If $\tilde{r} \geq
10$, from \eqref{II_2} and $\lambda<2\sqrt{t}$, we have $\tilde{r}\sim\lambda$.
Hence,  by letting $\lambda'=\lambda/\sqrt{1+\lambda^2}$ and using \eqref{beans0-l} with $\alpha = 0$, it follows that
\begin{equation*}
\begin{split}
II_2^1&\le\int_{\{\lambda<2\sqrt{t}:|\frac{t\lambda}{\sqrt{t^2+t\lambda^2}}-\tilde{r}|\leq 1\}}\lambda^{n-1}(1+\tilde{r}\lambda)^{-\frac{n-1}2}d\lambda
\\&\leq C \sqrt{t}\int_{\{\lambda<2:|\frac{\lambda}{\sqrt{1+\lambda^2}}-\frac{\tilde{r}}{\sqrt{t}}|\leq 1/\sqrt{t}\}}d\lambda\\&
\leq C \sqrt{t} \int_{\{\lambda<2:|\lambda'-\frac{\tilde{r}}{\sqrt{t}}|\leq 1/\sqrt{t}\}}(1+\lambda^2)^{3/2}d\lambda'\leq C.
\end{split}
\end{equation*}
Now we consider $II_2^2$. We estimate \begin{equation*}
\begin{split}
II_2^2\leq&\sum_{m\geq1}\Big|\int_0^\infty L^{N}
\big(e^{i(\sqrt{t^2+t\lambda^2}-
\tilde{r}\lambda)}\big)\Big[\chi_0(t^{-1/2}\lambda)\lambda^{n-1}a(t^{-1/2}\lambda,z,z')\\&\qquad \qquad\big(1-\varphi_0(\lambda)\big)\varphi\big(2^{-m}(\frac{t\lambda}{\sqrt{t^2+t\lambda^2}}-\tilde{r})\big)\big( 1 - \varphi_0(8\tilde{r} \lambda) \big) \Big] \, d\lambda\Big|.
\end{split}
\end{equation*}
Let $$b(\lambda)=\lambda^{n-1}a(t^{-1/2}\lambda,z,z')\big(1-\varphi_0(\lambda)\big)\varphi(\frac{\frac{t\lambda}{\sqrt{t^2+t\lambda^2}}-\tilde{r}}{2^m})\big( 1 - \varphi_0(8\tilde{r} \lambda) \big),$$
then on the support of $b$ with $\lambda\geq 1/2$, we use \eqref{beans0-l} to obtain
$$|\partial_\lambda ^\alpha b|\leq C_{\alpha}\lambda^{n-1} (1+\tilde{r}\lambda)^{-(n-1)/2}.$$
Hence from the first inequality of \eqref{I-bp4'}, we obtain
\begin{equation*}
\begin{split}
|(L^*)^N [b(\lambda)]|&\leq C_N 2^{-mN}\lambda^{n-1} (1+\tilde{r}\lambda)^{-(n-1)/2}.
\end{split}
\end{equation*}
Therefore we use integration by parts to obtain
\begin{equation*}
\begin{split}
II_2^2 \leq C_N\sum_{m\geq1}2^{-mN}\int_{\{\lambda<2\sqrt{t},|\frac{t\lambda}{\sqrt{t^2+t\lambda^2}}-\tilde{r}|\sim 2^m\}}\lambda^{n-1}(1+\tilde{r}\lambda)^{-\frac{n-1}2}d\lambda.
\end{split}
\end{equation*}
If $\tilde{r}\leq 2^{m+1}$, since $|\frac{t\lambda}{\sqrt{t^2+t\lambda^2}}-\tilde{r}|\sim 2^m$, then $\lambda \leq 2^{m+2}$. One has
\begin{equation*}
\begin{split}
II_2^2\leq C_N\sum_{m\geq 1}2^{-mN}2^{(m+2)n}\leq C.
\end{split}
\end{equation*}
If $\tilde{r}\geq 2^{m+1}$, we
have $\lambda\sim \tilde{r}$, thus we choose $N$ large enough such that
\begin{equation*}
\begin{split}
II_2^2\leq C_N t^{1/2}  \sum_{m\geq1}2^{-mN} \int_{\{\lambda<2:|\frac{\lambda}{\sqrt{1+\lambda^2}}-\frac{\tilde{r}}{\sqrt{t}}|\sim \frac{2^m}{\sqrt{t}}\}}d\lambda 
\leq C_N\sum_{m\geq1}2^{-mN}2^{m}\leq C
\end{split}
\end{equation*}
which concludes the proof of Lemma
\ref{dispersive-l}.\end{proof}\vspace{0.2cm}

\subsubsection{High frequency estimates} We now prove the microlocalized dispersive estimates for the high energy part.

\begin{proposition}[High energy estimates]\label{prop:hDispersive} Let $\leftidx{^{\sigma}}U_{j,k}(t) $ be defined as in \eqref{U-mic}.
Then there exists a constant $C$ independent of $t, z, z'$ for all
$j\geq 0, 0\leq k,k'\leq N$ and $0\leq \theta\leq 1$ such that
 the dispersive estimate
 \begin{equation}\label{h-}
 \big\|\leftidx{^{\sigma}}U_{j,k}(t)  (\leftidx{^{\sigma}}U_{j,k'}(\tau) )^*\big\|_{L^1\rightarrow L^\infty}\leq C
2^{j[(n+1+\theta)/2-2\sigma]}(2^{-j}+|t-\tau|)^{-(n-1+\theta)/2}\end{equation}
holds in each of the following cases:

\begin{itemize}
\item \label{h-near} $(k,k') \in J_{near}$ or $(k,k')=(0,k'), (k,0)$ and $t \neq \tau$; 

\item \label{h-nonout}  $(k,k')\in J_{non-out}$ and $t<\tau$;

\item \label{h-nonin}  $(k,k')\in J_{non-inc}$  and $\tau<t$.

\end{itemize}


\end{proposition}

\begin{remark} The dispersive inequalities \eqref{h-} in the last two cases are only needed to prove the double-endpoint inhomogeneous Strichartz estimate (and thus the endpoint); see Section \ref{sec:d-e-in}. 
\end{remark}

\begin{proof} The proof is similar to the proof of Proposition
\ref{prop:lDispersive}, but we need some modifications.

We first consider \eqref{h-} in the second and third cases. We only prove  \eqref{h-} in the third case (when $(k,k')\in J_{non-inc}$  and $\tau<t$) as the argument to prove \eqref{h-} in the
second case is analogous.  By \cite[Lemma 5.3]{HZ}, $\leftidx{^{\sigma}}U_{j,k}(t)  (\leftidx{^{\sigma}}U_{j,k'}(\tau) )^*$ is given by
\begin{equation}
\int_0^\infty e^{i(t-\tau){\sqrt{1+\lambda^2}}} \phi(2^{-j}\lambda)\langle \lambda\rangle^{-2\sigma}\big(Q_k(\lambda)
dE_{\sqrt{\Delta_g}}(\lambda)Q^*_{k'}(\lambda)\big)(z,z'),\quad \phi=\varphi^2.
\label{UiUjint}\end{equation} Then, under the assumption that $Q_k$ is not incoming-related to $Q_{k'}$, and $\tau<t$, we need to show that for $j\geq 0$ 
\begin{equation}\label{h-nonin'}
\begin{split}
&\Big|\int_0^\infty e^{i(t-\tau){\sqrt{1+\lambda^2}}} \phi(2^{-j}\lambda)\langle \lambda\rangle^{-2\sigma}\big(Q_k(\lambda)
dE_{\sqrt{\Delta_g}}(\lambda)Q^*_{k'}(\lambda)\big)(z,z')\Big|\\&\leq C 2^{j[(n+1+\theta)/2-2\sigma]}(2^{-j}+|t-\tau|)^{-(n-1+\theta)/2}.
\end{split}
\end{equation}
For the sake of simplicity, from \cite[Lemma 4.1, (4.4)-(4.6)]{ZZ2}, we only consider 
\begin{equation*}
\begin{split}
\int_0^\infty e^{i(t-\tau)\sqrt{1+\lambda^2}}\phi(2^{-j}\lambda)\langle \lambda\rangle^{-2\sigma}\int_{\R^\ell} e^{i\lambda\Phi(z,z',v)}\lambda^{n-1+\frac \ell 2}a(\lambda,z,z',v) dv d\lambda
\end{split}
\end{equation*}
where $\Phi(z,z',v)\geq c>0$ (due to the fact that $Q_k$
is not incoming-related to $Q_{k'}$), $0\leq \ell \leq n-1$ and $a$ is a smooth function such that 
$|(\lambda\partial_\lambda)^\alpha a|\leq C_\alpha $ (due to the fact that $a$ is compactly supported  in $v$). Let $\tilde{\Phi}(t,\tau;\lambda)=\frac{\Phi}{t-\tau}+\frac{2\lambda}{\sqrt{1+\lambda^2}}$ and define the operator 
$$
L = \frac1 i\Big(\Phi+\frac{2\lambda(t-\tau)}{\sqrt{1+\lambda^2}}\Big)^{-1}\partial_\lambda=\frac1i\left((t-\tau)\tilde{\Phi}(t,\tau;\lambda)\right)^{-1}\frac{\partial}{\partial \lambda},
$$
then $L^N(e^{i(t-\tau){\sqrt{1+\lambda^2}}+i\lambda \Phi(z,z',v)} )=e^{i(t-\tau){\sqrt{1+\lambda^2}}+i\lambda \Phi(z,z',v)} $. Let $L^*$ be its adjoint operator; integrating by parts yields
 \begin{equation*}
\begin{split}
&\Big|\int_0^\infty e^{i(t-\tau){\sqrt{1+\lambda^2}}} \phi(2^{-j}\lambda)\langle \lambda\rangle^{-2\sigma}\int_{\R^\ell} e^{i\lambda\Phi(z,z',v)}\lambda^{n-1+\frac \ell 2}a(\lambda,z,z',v) dv d\lambda\Big|
\\&\lesssim\int_{\R^\ell} \int_0^\infty \Big|(L^*)^N\left( \phi(2^{-j}\lambda)\langle \lambda\rangle^{-2\sigma} \lambda^{n-1+\frac \ell2}a(\lambda,z,z',v) \right) \Big| d\lambda dv.
\end{split}
\end{equation*}
Next, we claim that for any function $b(\lambda)$ which satisfies $|\partial_\lambda^\alpha b(\lambda)|\leq \lambda^{m-\alpha}$ (where  $\alpha\geq0$)
it holds that for any $N\geq 0$ and $t>\tau$, we have
\begin{equation}\label{I-bp0}
|(L^*)^N [b(\lambda)]|\leq C \lambda^{m-N} |(t-\tau)+c|^{-N},\quad \lambda\geq 1, c>0. 
\end{equation}
Indeed, we use the expression for  $\tilde{\Phi}(t,\tau;\lambda)$, an induction argument and the Leibniz rule to obtain:
\begin{equation*}
|(L^*)^N [b(\lambda)]|\leq C \lambda^{m-N} \sum_{j=0}^N\frac{(t-\tau)^j}{\big|(t-\tau)\tilde{\Phi}\big|^{N+j}}.
\end{equation*}
The expression  \eqref{I-bp0} follows due to $t-\tau>0$, $\Phi(z,z',v)\geq c>0$ and $\lambda\geq 1$.

Now, since $a$ is compactly supported in $v$, $t-\tau>0$ and $\Phi(z,z',v)\geq c>0$, we can apply \eqref{I-bp0} to obtain
 \begin{equation*}
\begin{split}
&\int_{\R^\ell} \int_0^\infty \Big|(L^*)^N\left( \phi(2^{-j}\lambda)\langle \lambda\rangle^{-2\sigma} \lambda^{n-1+\frac \ell2}a(\lambda,z,z',v) \right) \Big| d\lambda dv\\
&\lesssim|(t-\tau)+c|^{-N} \int_{2^{j-1}}^{2^{j+1}}  \lambda^{n-1+\frac \ell2-2\sigma-N} d\lambda 
\\&\lesssim 2^{j(n+\frac \ell2-2\sigma)}(2^j(|t-\tau|+c))^{-N}
\end{split}
\end{equation*}
which implies \eqref{h-nonin'} by choosing $N$ large enough.\vspace{0.2cm}

Finally, the proof of \eqref{h-} when the first condition holds  follows from Lemma  \ref{dispersive-h} below. 
\end{proof}

\begin{lemma}[Microlocalized dispersive estimates for high frequencies]\label{dispersive-h}
Suppose $(k,k')\in \{ 0, 1, \dots, N \}^2$ such that 
$Q_k(\lambda) dE_{\sqrt{\Delta_g}} Q_{k'}(\lambda)^*$ satisfies the conclusions of Proposition~\ref{prop:localized spectral measure}. Then for all integers
$j\geq0$ and for all $0\leq \theta\leq 1$, there exists a constant $C$ independent of $j$ and points $z,z'\in X$ such that 
\begin{equation}\label{disper}
\begin{split}
\Big|\int_0^\infty e^{it\sqrt{1+\lambda^2}} \phi(2^{-j}\lambda) \langle\lambda\rangle^{-2\sigma} \big(Q_k(\lambda)
&dE_{\sqrt{\Delta_g}}(\lambda)Q_{k'}^*(\lambda)\big)(z,z')
\Big|\\ &\leq C 2^{j[(n+1+\theta)/2-2\sigma]}\left(2^{-j}+|t|\right)^{-(n-1+\theta)/2}
\end{split}
\end{equation}
where $\phi\in\mathcal{C}_c^\infty ([\frac12,2])$.
\end{lemma}

\begin{proof}
Let $h=2^{-j}\leq 1$. From Proposition
\ref{prop:localized spectral measure}, we see
\begin{equation*}
\begin{split}
\Big|Q_k(\lambda)
dE_{\sqrt{\Delta_g}}(\lambda)Q_{k'}^*(\lambda)
\Big|\leq C \lambda^{n-1},
\end{split}
\end{equation*}
which directly implies  \eqref{disper}  if $|t|\leq h$.
From now on, we assume $|t|\geq h=2^{-j}$. By the scaling, this is a directly consequence of
\begin{equation}\label{disper'}
\begin{split}
\Big|\int_0^\infty e^{it\sqrt{h^2+\lambda^2}/h} \phi(\lambda)\langle\lambda/h\rangle^{-2\sigma}  &\big(Q_k
dE_{\sqrt{\Delta_g}}Q_{k'}^*\big)(\lambda/h, z,z')
\Big|\\&\leq C h^{2\sigma-(n-1)}(|t|/h)^{-\frac{n-1}2}(1+h|t|)^{-1/2}.
\end{split}
\end{equation}
Indeed if we could prove  \eqref{disper'}, then for $0\leq\theta\leq1$, we have 
\begin{equation*}
\begin{split}
&\Big|\int_0^\infty e^{it\sqrt{1+\lambda^2}} \phi(2^{-j}\lambda) \langle\lambda\rangle^{-2\sigma} \big(Q_k(\lambda)
dE_{\sqrt{\Delta_g}}(\lambda)Q_{k'}^*(\lambda)\big)(z,z')
\Big|\\ &\leq C 2^{j[(n+1)/2-2\sigma]}|t|^{-(n-1)/2}\left(1+2^{-j}|t|\right)^{-1/2}\\&\leq C 2^{j[(n+1+\theta)/2-2\sigma]}(2^{-j}+|t|)^{-(n-1+\theta)/2}(2^{-j}|t|)^{\frac{\theta}2}\left(1+2^{-j}|t|\right)^{-1/2}
\end{split}
\end{equation*}
which implies \eqref{disper}. Now we prove \eqref{disper'}.  Let $r=d(z,z')$, by using Proposition
\ref{prop:localized spectral measure}, we write
\begin{equation}\label{4.4}
\begin{split}
&\Big|\int_0^\infty e^{it\sqrt{h^2+\lambda^2}/h} \phi(\lambda) \langle\lambda\rangle^{-2\sigma} \big(Q_k
dE_{\sqrt{\Delta_g}}Q_{k'}^*\big)(\lambda/h, z,z')
d\lambda\\
&=\sum_\pm \int_0^\infty e^{it\sqrt{h^2+\lambda^2}/h}e^{\pm
ir\lambda/h}\phi(\lambda)\langle\lambda/h\rangle^{-2\sigma} (\lambda/h)^{n-1}a_\pm(\lambda/h,z,z')d\lambda\\&\qquad\qquad+\int_0^\infty
e^{it\sqrt{h^2+\lambda^2}/h}\phi(\lambda)(\lambda/h)^{n-1}\langle\lambda/h\rangle^{-2\sigma} b(\lambda/h,z,z')d\lambda
\end{split}
\end{equation}
where $a_\pm$ satisfies 
\begin{equation*}
\big|\partial_\lambda^\alpha a_\pm(\lambda,z,z') \big|\leq C_\alpha
\lambda^{-\alpha}(1+\lambda d(z,z'))^{-\frac{n-1}2},
\end{equation*}
and therefore
\begin{equation}\label{beans0-h}
\Big|\partial_\lambda^\alpha \big(a_\pm(h^{-1}\lambda,z,z')\big)
\Big|\leq C_\alpha \lambda^{-\alpha}(1+h^{-1}\lambda
r)^{-\frac{n-1}2}.
\end{equation}
We first use \eqref{beans} to obtain
\begin{equation}\label{est-b}
\begin{split}
\Big|\big(\frac{d}{d\lambda}\big)^{N}\big(\phi(\lambda)(\lambda/h)^{n-1}\langle\lambda/h\rangle^{-2\sigma} & b(\lambda/h,z,z')\big)\Big|\\&\leq
C_N(\lambda/h)^{n-1-2\sigma}\lambda^{-N},\quad \forall N\in\mathbb{N}.
\end{split}
\end{equation}
Let $\delta$ be a small constant to be chosen later. Recall that  $\varphi\in
C_c^\infty([\frac12,2])$ and $\varphi_0(\lambda)=\sum_{m\leq 0}\varphi(2^{-m}\lambda)$ as in \eqref{dp}. Since $\phi=\varphi^2$, then $\text{supp}( \phi)\subset [1/2,2]$, thus $\langle\lambda/h\rangle^{-2\sigma} \lesssim h^{2\sigma}$ when $\lambda$ is on
the support of $\phi$. Then by using \eqref{est-b}, we have
\begin{equation*}
\begin{split}
\Big|\int_0^\infty e^{it\sqrt{h^2+\lambda^2}/h} \phi(\lambda)(\lambda/h)^{n-1}\langle\lambda/h\rangle^{-2\sigma} &b(\lambda/h,z,z')
\varphi_0(\frac{\lambda}{\delta})d\lambda\Big|\\&\leq
C h^{2\sigma}\int_0^{\delta}(\lambda/h)^{n-1} d\lambda\leq C h^{2\sigma+1}(\delta/h)^{n}.
\end{split}
\end{equation*}
We use \eqref{beans} and $N$-times  integration by parts to obtain
\begin{equation*}
\begin{split}
&\Big|\int_0^\infty e^{it\sqrt{h^2+\lambda^2}/h}
\sum_{m\geq1}\varphi(\frac{\lambda}{2^{m}\delta})\phi(\lambda)(\lambda/h)^{n-1}\langle\lambda/h\rangle^{-2\sigma} b(\lambda/h,z,z')
d\lambda\Big|\\
&\leq \sum_{m\geq 1}\Big|\int_0^\infty
\big(\frac{h\sqrt{h^2+\lambda^2}}{\lambda
t}\frac\partial{\partial\lambda}\big)^{N}\big(e^{it\sqrt{h^2+\lambda^2}/h}\big)
\varphi(\frac{\lambda}{2^m\delta})\phi(\lambda)(\lambda/h)^{n-1}\langle\lambda/h\rangle^{-2\sigma} b(\lambda/h,z,z')
d\lambda\Big|\\& \leq
C_N(|t|/h)^{-N} h^{2\sigma-(n-1)}\sum_{m\geq1}\int_{2^{m-1}\delta}^{2^{m+1}\delta}\lambda^{n-1-2N} d\lambda\leq
C_N(|t|/h)^{-N}h^{2\sigma-(n-1)}\delta^{n-2N}.
\end{split}
\end{equation*}

Choosing $\delta=(|t|/h)^{-\frac12}$ and noting $|t|\geq h$,  thus we have proved 
\begin{equation*}
\begin{split}
&\Big|\int_0^\infty e^{it\sqrt{h^2+\lambda^2}/h} \phi(\lambda)(\lambda/h)^{n-1}b(\lambda/h,z,z')
d\lambda\Big|\\&\leq C h^{2\sigma+1} (h|t|)^{-\frac{n}2}\leq C h^{2\sigma}(h|t|)^{-\frac{n-1}2}(h^{-1}|t|)^{-1/2}\\&\leq C h^{2\sigma}(|t|h)^{-\frac{n-1}2}(1+h|t|)^{-1/2}.
\end{split}
\end{equation*}
which implies \eqref{disper'}.

Next we consider the terms with $a_\pm$ in \eqref{4.4}. Without loss of generality, we consider $t\gg h$. Let $\Phi_{\pm}(\lambda, h, r,t)=\sqrt{h^2+\lambda^2}\pm \frac{\lambda r}{t}$,  
it suffices to show that there exists a constant $C$ independent of $r, t$ and $h$ such that
\begin{equation}\label{disper''}
\begin{split}
|I_h^\pm(t,r)|\leq C h^{2\sigma} (|t|/h)^{-\frac{n-1}2}(1+h|t|)^{-1/2}
\end{split}
\end{equation}
where
\begin{equation*}
\begin{split}
I_h^\pm(t,r):=\int_0^\infty e^{i\frac{t}{h}\Phi_\pm(\lambda, h, r, t)}\phi(\lambda)\lambda^{n-1}\langle\lambda/h\rangle^{-2\sigma} a_\pm(\lambda/h,z,z')d\lambda.
\end{split}
\end{equation*}
If $r<t/4$ or $r>2t$, a simple computation gives
$$|\partial_\lambda \Phi_\pm(\lambda, h, r, t)|=\left|\frac{\lambda}{\sqrt{h^2+\lambda^2}}\pm\frac r t\right|\geq 1/4. $$

Now, let $L=(\frac {it}h\partial_\lambda\Phi)^{-1}\partial_\lambda$ and let $L^*$ be its adjoint operator. Suppose 
that $b(\lambda)$ satisfies $|\partial_\lambda^\alpha b(\lambda)|\leq \lambda^{m-\alpha}$ (where $\alpha\geq0$).
Then we claim that for any $N\geq 0$
\begin{equation}\label{I-bp1}
|(L^*)^N [b(\lambda)]|\leq C \lambda^{m-N} \sum_{j=0}^N\frac{(t/h)^j}{\left|\frac {it}h\partial_\lambda\Phi\right|^{N+j}}.
\end{equation}
Indeed, as in the proof of \eqref{I-bp0}, this is a consequence of the Leibniz rule and an induction argument. Returning to $I_h^\pm$, an integration by parts argument combined with \eqref{I-bp1},  for $r<\frac t4$ or $r>2t$, leads to
\begin{equation*}
\begin{split}
|I_h^\pm(t,r)|\leq C h^{2\sigma}(|t|/h)^{-N},\quad \forall N\geq 0
\end{split}
\end{equation*}
which implies \eqref{disper''} since $t\geq h$ and $h\leq 1$. Therefore we only consider  the case $t\sim r$.
To consider $I_h^+(t,r)$,
we first note that 
$$|\partial_\lambda \Phi_+(\lambda, h, r, t)|=\left|\frac{\lambda}{\sqrt{h^2+\lambda^2}}+\frac r t\right|\geq 1/2. $$
By using the same stationary phase argument as above, we also obtain
\begin{equation*}
\begin{split}
|I_h^+(t,r)|\leq C h^{2\sigma} (|t|/h)^{-N},\quad \forall N\geq 0,
\end{split}
\end{equation*}
which implies \eqref{disper''} since $t\geq h$ and $h\leq 1$.
To estimate $I_h^-(t,r)$, we need the following Van der Corput lemma, see \cite[Proposition 2, Page 332]{Stein}.
\begin{lemma*}[Van der Corput]\label{van-der} Let $\phi$ be real-valued and smooth in $(a,b)$, and that $|\phi^{(k)}(x)|\geq1$ for all $x\in (a,b)$. Then
\begin{equation}
\left|\int_a^b e^{i\lambda\phi(x)}\psi(x)dx\right|\leq c_k\lambda^{-1/k}\left(|\psi(b)|+\int_a^b|\psi'(x)|dx\right)
\end{equation}
holds when (i) $k\geq2$ or (ii) $k=1$ and $\phi'(x)$ is monotonic. Here $c_k$ is a constant depending only on $k$.
\end{lemma*}

For $h\leq 1$ and $\lambda\sim 1$, one can check that 
$$|\partial^2_\lambda \Phi_-(\lambda, h, r, t)|=\left|\frac{h^2}{\sqrt{h^2+\lambda^2}}\right|\geq \frac{h^2}{100}. $$
By using the Van der Corput Lemma  with $\lambda=th$ and $r\sim t$, we show
\begin{equation*}
\begin{split}
|I_h^-(t,r)|&\leq C (|t|h)^{-1/2}\int_0^\infty \left|\frac{d}{d\lambda}\left(\phi(\lambda)\lambda^{n-1}\langle\lambda/h\rangle^{-2\sigma} a_{-}(\lambda/h,z,z')\right)\right| d\lambda\\
&\leq Ch^{2\sigma} (|t|h)^{-1/2}\int_{1/2}^2\lambda^{n-2} (1+\lambda r/h)^{-\frac{n-1}2} d\lambda\\&\leq Ch^{2\sigma}(|t|h)^{-1/2}(|t|/h)^{-\frac{n-1}2}.
\end{split}
\end{equation*}
On the other hand, since $t\sim r$, a rough estimate gives
\begin{equation}\label{rough}
\begin{split}
|I_h^-(t,r)|\leq \int_0^\infty \phi(\lambda)\lambda^{n-1}\langle\lambda/h\rangle^{-2\sigma}(1+\lambda r/h)^{-(n-1)/2} d\lambda \leq C h^{2\sigma}(|t|/h)^{-\frac{n-1}2}.
\end{split}
\end{equation}
Thus we have proved
 $$|I_h^-(t,r)|\lesssim h^{2\sigma}\min\{1, (|t|h)^{-1/2}\}(|t|/h)^{-\frac{n-1}2},
$$ 
which implies \eqref{disper''}.
\end{proof}

\section{Strichartz estimates for the free flow }\label{sec:4} 

In this section, we prove the Strichartz estimates \eqref{stri} in Theorem \ref{thm:Strichartz} when $V=0$.

\subsection{Abstract Strichartz estimates} 
To obtain the Strichartz estimates, we need a
variant of Keel-Tao's abstract Strichartz estimates. We thus start by recalling the ``abstract" Strichartz
estimates in the Lorentz space $L^{r,2}$.

\begin{proposition}\label{prop:semi}
Let $(X,\mathcal{M},\mu)$ be a finite measure space and
$U: \mathbb{R}\rightarrow B(L^2(X,\mathcal{M},\mu))$ be a weakly
measurable map satisfying, for some constants $C$, $\alpha,\gamma\geq0$ and
$\beta, h>0$,
\begin{equation}\label{md}
\begin{split}
\|U(t)\|_{L^2\rightarrow L^2}&\leq Ch^{\gamma},\quad t\in \mathbb{R},\\
\|U(t)U(\tau)^*f\|_{L^\infty}&\leq
Ch^{2\gamma} h^{-\alpha}(h+|t-\tau|)^{-\beta}\|f\|_{L^1}.
\end{split}
\end{equation}
Then for every pair $q,r\in[1,\infty]$ such that $(q,r,\beta)\neq
(2,\infty,1)$ and
\begin{equation*}
\frac{1}{q}+\frac{\beta}{r}\leq\frac\beta 2,\quad q\ge2,
\end{equation*}
there exists a constant $\tilde{C}$ depending only on $C$, $\beta$,
$q$ and $r$ such that
\begin{equation}\label{s-stri}
\Big(\int_{\R}\|U(t) u_0\|_{L^{r,2}}^q dt\Big)^{\frac1q}\leq \tilde{C}
\Lambda(h)\|u_0\|_{L^2}
\end{equation}
where $\Lambda(h)=h^{-(\alpha+\beta)(\frac12-\frac1r)+\frac1q+\gamma}$.
\end{proposition}

\begin{proof}
When $\gamma=0$, this was proved in \cite[Proposition 4.2]{ZZ2}  by following  Keel-Tao's \cite{KT} well known result. The same argument works also for $\gamma>0$, the only difference being in the interpolation constant. We also refer to \cite[Theorem 10.7]{Zworski}.
\end{proof}

\subsection{Homogeneous Strichartz estimates} In this subsection, we prove the homogeneous Strichartz estimates in \eqref{stri}, i.e. with $F=0$. Before doing this, we prove the following two propositions. 

\begin{proposition}\label{l-mic-Strichartz} Let $\leftidx{^{\sigma}}U_{\mathrm{low}, k}(t)$ be defined in \eqref{U-l-mic} and let $f\in L^2$. Then there exists a constant $C$ such that
\begin{equation}\label{l-mic-stri}
\|\leftidx{^{\sigma}}U_{\mathrm{low}, k}(t) f\|_{L^q_tL^{r,2}_z(\mathbb{R}\times X)}\leq C
\|f\|_{L^2(X)},
\end{equation}
where the pair $(q,r)\in [2,\infty]^2$ satisfies
\begin{equation}\label{l-admissible}
2/q\leq n(1/2-1/r),\quad n\geq3.
\end{equation}

\end{proposition}

\begin{proof}

By using Proposition \ref{energy'} and \eqref{l-near} in
Proposition \ref{prop:lDispersive}, we verify the estimates \eqref{md}
for $\leftidx{^{\sigma}}U_{\mathrm{low}, k}(t)$, where $\alpha=\gamma=0$, $\beta=n/2$ and
$h=1$. Therefore we apply \eqref{s-stri} of Proposition \ref{prop:semi} to obtain
\eqref{l-mic-stri}. 
\end{proof}

\begin{proposition}\label{h-mic-Strichartz} Let $\leftidx{^{\sigma}}U_{j,k}(t)$ be defined in \eqref{U-mic} and let $f\in L^2$. Then for $j\geq 0$, there exists a constant $C$ such that
\begin{equation}\label{h-mic-stri}
\|\leftidx{^{\sigma}}U_{j,k}(t)f\|_{L^q_tL^{r,2}_z(\mathbb{R}\times X)}\leq C
2^{j(s-\sigma)}\|f\|_{L^2(X)},
\end{equation}
where the pair $(q,r)\in [2,\infty]^2$ satisfies $(q,r)\in \Lambda_{s,\theta}$ with $s\geq0$ and $0\leq\theta\leq 1$.
\end{proposition}

\begin{proof}
By using \eqref{L2-est-h} in Proposition \ref{energy} and \eqref{h-near} in
Proposition \ref{prop:hDispersive}, we have the estimates \eqref{md}
for $U_{j,k}(t)$, where $\alpha=(n+1+\theta)/2$, $\beta=(n-1+\theta)/2$, $\gamma=\sigma$ and
$h=2^{-j}$. Then, for $s=(n+\theta)(\frac12-\frac1r)-\frac1q$, it follows from Proposition \ref{prop:semi} that
\begin{equation*}
\|U_{j,k}(t)f\|_{L^q_t(\R:L^r(X))}\lesssim
2^{j[(n+\theta)(\frac12-\frac1r)-\frac1q-\sigma]} \|f\|_{L^2(X)},
\end{equation*}
which proves \eqref{h-mic-stri}.
\end{proof}

We are now ready to prove the homogeneous estimates in \eqref{stri}. Without losing in generality, we assume $u_1=0$. Recall that if $u$ solves \eqref{equ:KG} and $U(t)=e^{it\sqrt{1+\Delta_g}}$, then 
\begin{equation}\label{sleq}
\begin{split}
u(t)
=\frac{U(t)+U(-t)}2 u_0.
\end{split}
\end{equation}
We only estimate $\|U(t) u_0\|_{L^q(\R;L^{r}(X))}$ since the other term follows from the same argument.
To this end, we first write
\begin{equation*}
U(t)=\sum_{k=0}^{N}\sum_{j\in\Z}U_{j,k}(t)=\sum_{k=0}^{N}U_{\mathrm{low},k}(t)+\sum_{k=0}^{N}\sum_{j\geq1}U_{j,k}(t),
\end{equation*}
where $U_{j,k}(t)$ and $U_{\mathrm{low},k}(t)$ are defined in \eqref{U-mic} and \eqref{U-l-mic} with $\sigma=0$ respectively.
Hence we can write
\begin{equation*}
U(t)u_0=\sum_{k=0}^{N}U_{\mathrm{low},k}(t)u_0+\sum_{k=0}^{N}\sum_{j\geq1}U_{j,k}(t)u_0.
\end{equation*}
Due to the finiteness of $N$,  it suffices to show
\begin{equation}\label{s-l}
\begin{split}
\|U_{\mathrm{low},k} u_0\|_{L^q(\R;L^{r,2}(X))}\leq C\|u_0\|_{L^2(X)}, \quad2/q\leq n(1/2-1/r),
\end{split}
\end{equation}
and for $2/q\leq (n-1+\theta)(1/2-1/r)$ and $0\leq\theta\leq 1$, 
\begin{equation}\label{s-h}
\big\|\sum_{j\geq1}U_{j,k}(t) u_0\big\|_{L^q_t(\R:L^{r,2}(X))}\lesssim
\| u_0\|_{H^s(X)},\quad s=(n+\theta)(\frac12-\frac1r)-\frac1q.
\end{equation}
The low energy estimate \eqref{s-l} directly follows from \eqref{l-mic-stri}.
For the high energy part, since
$q,r\geq2$,
we use the square-function
estimates \eqref{square} and Minkowski's inequality to obtain 
\begin{equation}\label{LP}
\big\|\sum_{j\geq1}U_{j,k}(t) u_0\big\|^2_{L^q_t(\R:L^{r,2}(X))}\lesssim 
\sum_{\ell\in \Z}\|\sum_{j\geq0}U_{j,k}(t) \varphi(2^{-\ell}\sqrt{\Delta_g})u_0\|^2_{L^q(\R;L^{r,2}(X))}.
\end{equation}
In addition,  we observe that
\begin{equation*}
\begin{split}
U_{j,k}(t)f&=\int_0^\infty
e^{it\sqrt{1+\lambda^2}}\varphi(2^{-j}\lambda)\tilde{\varphi}(2^{-j}\lambda) Q_k(\lambda)dE_{\sqrt{\Delta_g}}(\lambda)
f\\
&=\int_0^\infty
e^{it\sqrt{1+\lambda^2}}\varphi(2^{-j}\lambda)Q_k(\lambda)dE_{\sqrt{\Delta_g}}(\lambda)
\widetilde{\varphi}(2^{-j}\sqrt{\Delta_g})f
\end{split}
\end{equation*}
where $\widetilde{\varphi} \in C_0^\infty([\frac14,4])$, takes 
values in $[0,1]$ and is such that $\widetilde{\varphi}\varphi=\varphi$. Due to the fact that
$\widetilde{\varphi}(2^{-j}\sqrt{\Delta_g}) \varphi(2^{-\ell}\sqrt{\Delta_g})u_0$ vanishes if
$|\ell-j|\geq5$, we need to estimate
\begin{equation*}
\begin{split}
&\sum_{\ell\in \Z}\|\sum_{j\geq0}U_{j,k}(t) \varphi(2^{-\ell}\sqrt{\Delta_g})u_0\|^2_{L^q(\R;L^{r,2}(X))}\\
&= \sum_{\ell\in \Z}\big(\sum_{|\ell-j|\leq5}\|U_{j,k}(t) \varphi(2^{-\ell}\sqrt{\Delta_g})u_0\|_{L^q(\R;L^{r,2}(X))}\big)^2.
\end{split}
\end{equation*}
By Proposition \ref {h-mic-Strichartz} with $\sigma=0$, we have 
\begin{equation*}
\begin{split}
&\|U_{j,k}(t) \varphi(2^{-\ell}\sqrt{\Delta_g})u_0\|_{L^q_t(\R:L^{r,2}(X))}\\&\lesssim
2^{j[(n+\theta)(\frac12-\frac1r)-\frac1q]} \|\varphi(2^{-\ell}\sqrt{\Delta_g})u_0\|_{L^2(X)}.
\end{split}\end{equation*}
Therefore, for $s=(n+\theta)(\frac12-\frac1r)-\frac1q$, we obtain
\begin{equation}\label{s-h'}
\begin{split}
&\big\|\sum_{j\geq0}U_{j,k}(t) u_0\big\|^2_{L^q_t(\R:L^{r,2}(X))}\\&\lesssim
\sum_{|j-\ell|\leq 5}2^{2j[(n+\theta)(\frac12-\frac1r)-\frac1q]} \|\varphi(2^{-\ell}\sqrt{\Delta_g})u_0\|^2_{L^2(X)}
\lesssim \|u_0\|^2_{H^s(X)}.
\end{split}
\end{equation}
This gives \eqref{s-h}.
Therefore we obtain the Strichartz estimates with $u_1=F=0$
\begin{equation*}
\|u\|_{L^q(\R;L^{r,2}(X))}\leq C\|u_0\|_{H^s(X)}.
\end{equation*}

\subsection{Inhomogeneous Strichartz estimates}\label{sec:inh}
We now turn to the inhomogeneous Strichartz estimates in \eqref{stri}, which we prove via the $TT^*$-method and Christ-Kiselev lemma in \cite{CK}. Recall that $U(t)=e^{it\sqrt{1+\Delta_g}}: L^2\rightarrow L^2$; we have already proved that the inequality
\begin{equation*}
\|U(t)u_0\|_{L^q_tL^{r,2}_z}\lesssim\|u_0\|_{H^s}
\end{equation*} holds for all $(q,r,s)$ satisfying \eqref{adm} and \eqref{scaling}.
For $s\geq0$ and $(q,r)\in \Lambda_{s,\theta}$ satisfying \eqref{adm} and \eqref{scaling},
we define the operator ${\bf T}_s$ by
\begin{equation}\label{Ts}
\begin{split}
{\bf T}_s: L^2_z&\rightarrow L^q_tL^{r,2}_z,\quad f\mapsto (1+\Delta_g)^{-\frac
s2}e^{it\sqrt{1+\Delta_g}}f.
\end{split}
\end{equation}
By duality, for $(\tilde{q},\tr)\in \Lambda_{1-s,\ttheta}$, we have
\begin{equation}\label{Ts*}
\begin{split}
{\bf T}^*_{1-s}: L^{\tilde{q}'}_tL^{\tilde{r}',2}_z\rightarrow L^2,\quad
F(\tau,z)&\mapsto \int_{\R}(1+\Delta_g)^{\frac
{s-1}2}e^{-i\tau\sqrt{1+\Delta_g}}F(\tau)d\tau.
\end{split}
\end{equation}
Therefore we obtain
\begin{equation*}
\Big\|\int_{\R}U(t)U^*(\tau)(1+\Delta_g)^{-\frac12}F(\tau)d\tau\Big\|_{L^q_tL^{r,2}_z}
=\big\|{\bf T}_s {\bf T}^*_{1-s}F\big\|_{L^q_tL^{r,2}_z}\lesssim\|F\|_{L^{\tilde{q}'}_tL^{\tilde{r}',2}_z}.
\end{equation*}
 By the
Christ-Kiselev Lemma, if $q>\tilde{q}'$, we thus obtain
\begin{equation}\label{non-inhomgeneous}
\begin{split}
\Big\|\int_{\tau<t}\frac{\sin{(t-\tau)\sqrt{1+\Delta_g}}}
{\sqrt{1+\Delta_g}}F(\tau)d\tau\Big\|_{L^q_tL^{r,2}_z}\lesssim\|F\|_{L^{\tilde{q}'}_t{L}^{\tilde{r}',2}_z}.
\end{split}
\end{equation}

On the other hand, we have the following
\begin{lemma}\label{lem:qq'} 
If $(q,r)\in \Lambda_{s,\theta}$ and  $(\tilde{q},\tilde{r})\in \Lambda_{1-s,\ttheta}$ with $0\leq\theta,\ttheta\leq1$, then $q>\tilde{q}'$.
\end{lemma}

\begin{proof} By the definition of the set $\Lambda_{s,\theta}$, if $(q,r)\in \Lambda_{s,\theta}$ and $(\tilde{q},\tilde{r})\in \Lambda_{1-s,\ttheta}$, we have
\begin{equation*}
\begin{split}
1-s=(n+\ttheta)(\frac12-\frac1\tr)-\frac1{\tilde q}, \quad \frac2{\tilde{q}}\leq (n+\ttheta)(\frac12-\frac1\tr)\\
s=(n+\theta)(\frac12-\frac1r)-\frac1{ q}, \quad \frac2q\leq (n+\theta)(\frac12-\frac1r)
\end{split}
\end{equation*}
then we have
\begin{equation*}
\begin{split}
1-\frac1{q}\big(\frac{2(n+\theta)}{n-1+\theta}-1\big)\geq 1-s=(n+\ttheta)(\frac12-\frac1\tr)-\frac1{\tilde q}\geq\frac1{\tilde q}\big(\frac{2(n+\ttheta)}{n-1+\ttheta}-1\big)
\end{split}
\end{equation*}
which implies
\begin{equation*}
\begin{split}
1-\frac1{q}\big(1+\frac{2}{n-1+\theta}\big)\geq\frac1{\tilde q}\big(1+\frac{2}{n-1+\ttheta}\big),
\end{split}
\end{equation*}
hence
\begin{equation*}
\begin{split}
\frac1{\tilde{q}'}-\frac1q \geq \frac{2}{\tilde{q}(n-1+\ttheta)}+\frac{2}{q(n-1+\theta)}>0.
\end{split}
\end{equation*}
\end{proof}

Hence we have proved all inhomogeneous Strichartz estimates  for
$(q,r)\in \Lambda_{s,\theta}$ and  $(\tilde{q},\tilde{r})\in \Lambda_{1-s,\ttheta}$ with $0\leq\theta,\ttheta\leq1$.
Therefore, we conclude that:
\begin{proposition}\label{Str-L0} Let $s\geq0$, $(q,r)\in \Lambda_{s,\theta}$ and  $(\tilde{q},\tilde{r})\in \Lambda_{1-s,\ttheta}$ with $0\leq\theta,\ttheta\leq1$ and let $u$ be the solution to
\begin{equation}\label{leq}
\partial_{t}^2u+\Delta_g u+u=F, \quad u(0)=u_0,
~\partial_tu(0)=u_1,
\end{equation}
the following Strichartz estimates hold: 
\begin{equation}\label{Str-L0-est}
\|u(t,z)\|_{L^q(\R;L^{r,2}(X))}\leq C\left(\|u_0\|_{H^s(X)}+\|u_1\|_{ H^{s-1}(X)}+\|F\|_{L^{\tilde{q}'}(\R;L^{\tilde{r}',2}(X))}\right).
\end{equation}

\end{proposition}

\begin{remark} This result concludes the full set of global-in-time Strichartz inequalities when $V=0$.
Hence, by the embedding inequality for Lorentz spaces, we obtain \eqref{stri}.
\end{remark}

\begin{remark}\label{rem:equiv} The Sobolev norm  in \eqref{Str-L0-est} is equivalent to the Sobolev norm in \eqref{V-stri} and \eqref{V-stri'} when $0\leq s\leq 1$ due to the argument of  \cite[Proposition 1.3, Corollary 1.4]{BPST}.

\end{remark}

\section{Inhomogeneous Strichartz estimates with $q=\tilde{q}=2$}\label{sec:d-e-in} 
To prove the Strichartz estimates on the board line $q=2$ (i.e. \eqref{V-stri'}),  we need the double endpoint  inhomogeneous Strichartz estimates, that is, \eqref{Str-L0-est} with $q=\tilde{q}=2$.
However, the above argument breaks down here due to the failure of the Christ-Kiselev lemma.  We follow the argument in Keel-Tao \cite{KT} to overcome this obstacle; notice that here
we need to face the additional difficulty due to the the lack of the usual dispersive estimates, that  are known to fail in presence of conjugate points in the space, for example, see \cite{HW}. Nevertheless, we can overcome this by following the argument in \cite{HZ}.

\begin{proposition}\label{prop:inh} Let $0\leq \theta,\ttheta \leq 1$, and let $r=\frac{2(n-1+\theta)}{n-3+\theta}$, $\tr=\frac{2(n-1+\ttheta)}{n-3+\ttheta}$, the following 
inhomogeneous Strichartz inequalities hold
\begin{equation}\label{inh}
\begin{split}
\Big\|(1+\Delta_g)^{-\frac{\alpha+\talpha} 2}\int_{\tau<t}\frac{\sin{\big((t-\tau)\sqrt{1+\Delta_g}\big)}}
{\sqrt{1+\Delta_g}}F(\tau)d\tau\Big\|_{L^2_tL^{r,2}_z}\lesssim \|F\|_{L^{2}_t{L}^{\tr',2}_z},
\end{split}
\end{equation}
where
\begin{equation}\label{alpha}
\alpha=\frac1{n-1+\theta}, \quad \talpha=\frac1{n-1+\ttheta}.
\end{equation}
\end{proposition}

\begin{remark}
This inhomogeneous inequalities are not included in the above estimates \eqref{Str-L0-est} since if $q=\tilde{q}=2$, then at least, one of the conditions $(q,r)\in \Lambda_{s,\theta}$, $(\tilde{q},\tilde{r})\in \Lambda_{1-s,\ttheta}$ is not fulfilled (see Lemma \ref{lem:qq'}).
\end{remark}

\begin{proof} The proof directly follows from the following proposition.
\end{proof}

\begin{proposition}\label{prop:dinh}  Let $\theta,\ttheta, r,\tr, \alpha,\talpha$ be in Proposition \ref{prop:inh}.
Then the following inequalities hold:

$\bullet$ Low frequencies estimates
\begin{equation}\label{l-dinh}
\begin{split}
\Big\|\int_{\tau<t}\frac{\sin{\big((t-\tau)\sqrt{1+\Delta_g}\big)}}
{\sqrt{1+\Delta_g}}\varphi_0(\sqrt{\Delta_g})F(\tau)d\tau\Big\|_{L^2_tL^{r}_z}\lesssim \|F\|_{L^{2}_t{L}^{\tr'}_z},
\end{split}
\end{equation}

$\bullet$ High frequencies estimates
\begin{equation}\label{h-dinh}
\begin{split}
\Big\| \int_{\tau<t}\frac{\sin{\big((t-\tau)\sqrt{1+\Delta_g}\big)}}
{\sqrt{1+\Delta_g}}(1-\varphi_0)(\sqrt{\Delta_g})&F(\tau)d\tau\Big\|_{L^2_tL^{r}_z} \lesssim \|(1+\Delta_g)^{\frac{\alpha+\talpha}{2}} F\|_{L^{2}_t{L}^{\tr'}_z},
\end{split}
\end{equation}
where $\varphi_0\in C_c^\infty([0,\infty)$ such that $\varphi_0(\lambda) = 1$ for $\lambda \leq 1$ and vanishes when $\lambda\geq2$.  

\end{proposition}

\begin{proof} 
Recalling that $U(t)=e^{it\sqrt{1+\Delta_g}}$, we have
\begin{equation}\label{sin-UU}
\begin{split}
\frac{\sin{\big((t-\tau)\sqrt{1+\Delta_g}\big)}}
{\sqrt{1+\Delta_g}}=\langle \sqrt{\Delta_g} \rangle^{-1}(U(t)U(\tau)^*-U(-t)U(-\tau)^*)/2i.
\end{split}
\end{equation}

We first prove \eqref{l-dinh}. We only estimate the term involving $\varphi_0(\sqrt{\Delta_g})U(t)U(\tau)^*$ since the other one follows from the same argument. To this end, we write
\begin{equation*}
\begin{split}
\langle \sqrt{\Delta_g} \rangle^{-1}\varphi_0(\sqrt{\Delta_g})U(t)U(\tau)^*
=\leftidx{^{\sigma}}{\tilde U}_{\mathrm{low}}(t)(\leftidx{^{\sigma}}{\tilde U}_{\mathrm{low}}(\tau))^*,
\end{split}
\end{equation*}
where 
\begin{equation}\label{U-l}
\leftidx{^{\sigma}}{\tilde U}_{\mathrm{low}}(t) = \int_0^\infty e^{it\sqrt{1+\lambda^2}} \varphi_0^{1/2}(\lambda)\langle \lambda\rangle^{-\sigma}
dE_{\sqrt{\Delta_g}}(\lambda),\quad \sigma=1/2.
\end{equation}
For the partition of identity operator $Q_k(\lambda)$ in Proposition \ref{prop:localized spectral measure}, we further define
\begin{equation}\label{tU-l-mic}
\begin{gathered}
\leftidx{^{\sigma}}{\tilde U}_{\mathrm{low},k}(t)  = \int_0^\infty e^{it\sqrt{1+\lambda^2}} \varphi^{1/2}_0(\lambda)\langle \lambda\rangle^{-\sigma}Q_{k}(\lambda)
dE_{\sqrt{\Delta_g}}(\lambda), \quad 0\leq k\leq N.
\end{gathered}
\end{equation}
Compared with $\leftidx{^{\sigma}}{U}_{\mathrm{low},k}(t)$ defined in \eqref{U-l-mic}, we replace 
$\varphi_0$ by $\varphi_0^{1/2}$ respectively. The differences are harmless in view of obtaining Proposition \ref{prop:lDispersive}, so we drop off the tilde from now on.

To prove \eqref{l-dinh}, due to the finiteness of $N$, it is enough to show the bilinear form estimate
\begin{equation}\label{lTFG}
|T(F,G)|\lesssim  \|F\|_{L^{2}_tL^{\tr',2}_z}\|G\|_{L^{2}_tL^{r',2}_z},
\end{equation}
where $T(F,G)$ is the bilinear form
\begin{equation*}
T(F,G)=\iint_{\tau<t}\langle \leftidx{^{\sigma}}U_{\mathrm{low},k}(t)(\leftidx{^{\sigma}}U_{\mathrm{low},k'}(\tau))^*F(\tau), G(t)\rangle_{L^2}~ d\tau dt.
\end{equation*}
We need the following bilinear estimates

\begin{lemma}\label{l-either} Let $\leftidx{^{\sigma}}U_{\mathrm{low},k}(t)$ be defined as in \eqref{U-l-mic} with $\sigma=1/2$, then for each pair $(k,k')\in \{ 0, 1, \dots, N \}^2 $
there exists a constant $C$ such that,  either
\begin{equation}\label{l-bilinear:s<t}
\iint_{\tau<t}\langle  \leftidx{^{\sigma}}U_{\mathrm{low},k}(t)(\leftidx{^{\sigma}}U_{\mathrm{low},k'}(\tau))^*F(\tau), G(t)\rangle_{L^2}~ d\tau dt\leq C 
\|F\|_{L^2_\tau L^{\tr',2}_z}\|G\|_{L^2_tL^{r',2}_z},
\end{equation}
or
\begin{equation}\label{l-bilinear:s>t}
\iint_{\tau>t}\langle  \leftidx{^{\sigma}}U_{\mathrm{low},k}(t)(\leftidx{^{\sigma}}U_{\mathrm{low},k'}(\tau))^*F(\tau), G(t)\rangle_{L^2}~ d\tau dt\leq C 
\|F\|_{L^2_\tau L^{\tr',2}_z}\|G\|_{L^2_tL^{r',2}_z}.
\end{equation}
\end{lemma}

We postpone the proof at the end of this section. Note that $\sigma$ plays no role in the low frequency estimates. In Proposition \ref{l-mic-Strichartz}, we have proved 
\begin{equation}\label{s-l'}
\begin{split}
\|\leftidx{^{\sigma}} U_{\mathrm{low},k} u_0\|_{L^q(\R;L^{r,2}(X))}\leq C\|u_0\|_{L^2(X)}, \quad2/q\leq n(1/2-1/r).
\end{split}
\end{equation}
In particular, we take $(q,r)=(2, \frac{2(n-1+\theta)}{n-3+\theta})$ and $(\tilde{q},\tr)=(2, \frac{2(n-1+\ttheta)}{n-3+\ttheta})$.  By duality, we have
\begin{equation*}
\Big\|\int_{\R}\leftidx{^{\sigma}}U_{\mathrm{low},k}(t)(\leftidx{^{\sigma}}U_{\mathrm{low},k'}(\tau))^*F(\tau)d\tau\Big\|_{L^2_tL^{r,2}_z}\lesssim \|F\|_{L^{2}_\tau L^{\tr',2}_z}\end{equation*}
for all  $0\leq k,k'\leq N$. For all $0\leq k,k'\leq N$, thus it follows that 
\begin{equation}\label{R2}
\iint_{\R^2}\langle \leftidx{^{\sigma}}U_{\mathrm{low},k}(t)(\leftidx{^{\sigma}}U_{\mathrm{low},k'}(\tau))^*F(\tau), G(t)\rangle_{L^2}~ d\tau dt\leq
C \|F\|_{L^2_\tau L^{\tr',2}_z}\|G\|_{L^2_tL^{r',2}_z}.
\end{equation}
Hence for every pair $(k,k')$, we have by \eqref{l-bilinear:s<t} or subtracting \eqref{l-bilinear:s>t} from \eqref{R2} 
\begin{equation*}
\iint_{\tau<t}\langle \leftidx{^{\sigma}}U_{\mathrm{low},k}(t)(\leftidx{^{\sigma}}U_{\mathrm{low},k'}(\tau))^*F(\tau), G(t)\rangle_{L^2}~ d\tau dt\leq C 
\|F\|_{L^2_\tau L^{\tr',2}_z}\|G\|_{L^2_tL^{r',2}_z}.
\end{equation*}
Hence we obtain \eqref{lTFG} once we have proved Lemma \ref{l-either}.\vspace{0.2cm}

We next prove \eqref{h-dinh}. By using \eqref{sin-UU} again, we only estimate the term involving $(1-\varphi_0)(\sqrt{\Delta_g})U(t)U(\tau)^*$ since the other term follows the same argument. To this end, we write
\begin{equation*}
\begin{split}
\langle \sqrt{\Delta_g} \rangle^{-(1+\alpha+\talpha)}(1-\varphi_0)(\sqrt{\Delta_g})U(t)U(\tau)^*
=\sum_{j\geq 1} \leftidx{^{\sigma}}{\tilde U}_{j}(t)(\leftidx{^{\tsigma}}{\tilde U}_{j}(\tau))^*,
\end{split}
\end{equation*}
where $\sigma=1/2+\alpha$, $\tsigma=1/2+\talpha$ and
\begin{equation}\label{U-h}
\leftidx{^{\sigma}}{\tilde U}_{j}(t) = \int_0^\infty e^{it\sqrt{1+\lambda^2}} \varphi^{1/2}(2^{-j}\lambda)\langle \lambda\rangle^{-\sigma}
dE_{\sqrt{\Delta_g}}(\lambda).
\end{equation}
Using $Q_k(\lambda)$ in Proposition \ref{prop:localized spectral measure}, we further define
\begin{equation}\label{tU-mic}
\begin{gathered}
\leftidx{^{\sigma}}{\tilde U}_{j,k}(t) = \int_0^\infty e^{it\sqrt{1+\lambda^2}} \varphi^{1/2}(2^{-j}\lambda)\langle \lambda\rangle^{-\sigma}Q_{k}(\lambda)
dE_{\sqrt{\Delta_g}}(\lambda),~j\geq0, \,  0\leq k\leq N.
\end{gathered}
\end{equation}
Thus to prove \eqref{h-dinh}, it suffices to show
\begin{equation}\label{h-dinh'}
\begin{split}
\sum_{0\leq k,k'\leq N}\Big\|\sum_{j\geq 1}\int_{\tau<t} \leftidx{^{\sigma}}{\tilde U}_{j,k}(t)(\leftidx{^{\tsigma}}{\tilde U}_{j,k'}(\tau))^* F(\tau)d\tau\Big\|_{L^2_tL^{r,2}_z}\lesssim \|F\|_{L^{2}_t{L}^{\tr',2}_z}.
\end{split}
\end{equation}
Compared with $\leftidx{^{\sigma}}{U}_{j,k}(t)$ defined in \eqref{U-mic}, we replace 
$\varphi$ by $\varphi^{1/2}$. The differences are harmless to obtain Proposition \ref{prop:hDispersive}, so we also drop off the tilde on $U_{j,k}$ from now on.

The argument is almost the same as the above argument for low frequencies.
In order to show \eqref{h-dinh}, by using the Littlewood-Paley theory in Lemma \ref{prop:square}, it suffices to show the bilinear form estimate
\begin{equation}\label{hTFG}
|T_{j}(F,G)|\leq C \|F_j\|_{L^{2}_tL^{\tr',2}_z}\|G_j\|_{L^{2}_tL^{r',2}_z},
\end{equation}
where $F_j=\varphi(2^{-j}\sqrt{\Delta_g})F$, $G_j=\varphi(2^{-j}\sqrt{\Delta_g})G$ and $T_j(F,G)$ is the bilinear form
\begin{equation*}
T_{j}(F,G)=\iint_{s<t}\langle \leftidx{^{\sigma}}{ U}_{j,k}(t)(\leftidx{^{\tsigma}}{ U}_{j,k'}(\tau))^* F(\tau), G(t)\rangle_{L^2}~ d\tau dt
\end{equation*}
where $\leftidx{^{\sigma}} U_{j,k}$ defined in \eqref{U-mic}. In Proposition \ref{h-mic-Strichartz}, we have proved
\begin{equation*}
\|\leftidx{^{\sigma}} U_{j,k}f\|_{L^2_t(\R:L^{r,2}(X))}\lesssim 2^{j[(n+\theta)(\frac12-\frac1r)-\frac12-\sigma]}\|f\|_{L^2(X)}
\end{equation*}
and
\begin{equation*}
\|\leftidx{^{\tsigma}} U_{j,k}f\|_{L^2_t(\R:L^{\tr,2}(X))}\lesssim 2^{j[(n+\ttheta)(\frac12-\frac1\tr)-\frac12-\tsigma]}\|f\|_{L^2(X)}.
\end{equation*}
Recall $r=\frac{2(n-1+\theta)}{n-3+\theta}$ and $\sigma=\frac12+\frac1{n-1+\theta}$, thus $(n+\theta)(\frac12-\frac1r)-\frac12-\sigma=0$ which also occurs in the case of $\tr$ and $\tsigma$. 
By duality, for all $0\leq k,k'\leq N$, we have
\begin{equation*}
\Big\|\int_{\R}\leftidx{^{\sigma}}{ U}_{j,k}(t)(\leftidx{^{\tsigma}}{ U}_{j,k'}(\tau))^*F(\tau)d\tau\Big\|_{L^2_tL^{r,2}_z}\lesssim \|F\|_{L^{2}_\tau L^{\tr',2}_z}.
\end{equation*}
which implies
\begin{equation}\label{h-R2}
\iint_{\R^2}\langle \leftidx{^{\sigma}}{ U}_{j,k}(t)(\leftidx{^{\tsigma}}{ U}_{j,k'}(\tau))^*F(\tau), G(t)\rangle_{L^2}~ d\tau dt\leq
C \|F\|_{L^2_\tau L^{\tr',2}_z}\|G\|_{L^2_tL^{r',2}_z}.
\end{equation}

To show \eqref{hTFG}, we need the following bilinear estimates

\begin{lemma}\label{h-either} Let $\theta,\ttheta, r,\tr, \alpha,\talpha$ be in Proposition \ref{prop:inh} and let $\leftidx{^{\sigma}}{ U}_{j,k}(t)$ and
$\leftidx{^{\tsigma}}{ U}_{j,k}(t)$ be defined as in \eqref{U-mic} with $\sigma=\frac12+\alpha$ and $\tsigma=\frac12+\talpha$.  Then for each pair $(k,k')\in \{ 0, 1, \dots, N \}^2 $
there exists a constant $C$ such that, for each $j$, either
\begin{equation}\label{bilinear:s<t}
\iint_{\tau<t}\langle \leftidx{^{\sigma}}{ U}_{j,k}(t)(\leftidx{^{\tsigma}}{ U}_{j,k'}(\tau))^*F(\tau), G(t)\rangle_{L^2}~ d\tau dt\leq C
\|F\|_{L^2_\tau L^{\tr',2}_z}\|G\|_{L^2_tL^{r',2}_z},
\end{equation}
or
\begin{equation}\label{bilinear:s>t}
\iint_{\tau>t}\langle \leftidx{^{\sigma}}{ U}_{j,k}(t)(\leftidx{^{\tsigma}}{ U}_{j,k'}(\tau))^*F(\tau), G(t)\rangle_{L^2}~ d\tau dt\leq C 
\|F\|_{L^2_\tau L^{\tr',2}_z}\|G\|_{L^2_tL^{r',2}_z}.
\end{equation}
\end{lemma}
We postpone the proof for a moment. Hence for every pair $(k,k')$, we have by \eqref{bilinear:s<t} or subtracting \eqref{bilinear:s>t} from \eqref{h-R2} 
\begin{equation*}
\iint_{\tau<t}\langle \leftidx{^{\sigma}}{ U}_{j,k}(t)(\leftidx{^{\tsigma}}{ U}_{j,k'}(\tau))^*F(\tau), G(t)\rangle_{L^2}~ d\tau dt\leq C
\|F\|_{L^2_\tau L^{\tr',2}_z}\|G\|_{L^2_tL^{r',2}_z}.
\end{equation*}
Therefore, proving Lemma \ref{h-either} completes the proof of \eqref{h-dinh}.
\end{proof}

\begin{proof}[Proof of Lemma~\ref{l-either} and Lemma~ \ref{h-either}] By \eqref{cl-N}, we need to consider three cases. We first prove Lemma~\ref{l-either}. In the case that $(k,k') \in J_{near}$ or $(k,k')=(k,0)$ or $(k,k')=(0,k')$,
we have the dispersive estimate \eqref{l-near}. We 
apply the argument of \cite[Sections 4--7]{KT} to obtain
\eqref{l-bilinear:s<t}. If $(k,k') \in J_{non-out}$, we
obtain \eqref{l-bilinear:s<t} adapting the argument in \cite{KT}
due to the dispersive estimate \eqref{l-nonout} when $\tau > t$.
Finally, in the case that $(k,k') \in J_{non-inc}$, we obtain
\eqref{l-bilinear:s>t} since we have the dispersive estimate
\eqref{l-nonin} for $\tau< t$. We mention here that we have sharpened the inequality to the Lorentz norm by the interpolation as remarked in \cite[Section 6 and Section 10]{KT}.

The same argument works for Lemma~\ref{h-either} by using the dispersive estimates \eqref{h-near}, \eqref{h-nonout} and \eqref{h-nonin}.
\end{proof}

\section{Local-smoothing estimates for Klein-Gordon}\label{sec:6}

In order to prove Strichartz estimates for the Klein-Gordon equation associated to $\LL_V$, we need a global-in-time local-smoothing estimate, that we prove in this section.
The local-smoothing estimate allows us to use the (standard) perturbation argument to treat the potential term. We stress the fact that we will need to deal separately with low and high frequencies when applying this method to the Klein-Gordon equation.

\subsection{Preliminaries: basic analysis on metric cone}
We consider the operator $\widetilde{\Delta}_h:=\Delta_h+V_0(y)$ with $V_0(y)\in\mathcal{C}^\infty(Y)$, on the closed Riemannian manifold $Y$. It is well known that 
the eigenvalues $\{\lambda_j\}_{j=0}^\infty$ of the Schr\"odinger operator $\widetilde{\Delta}_h$ form a discrete set and that the sequence 
$$\lambda_0<\lambda_1<\cdots<\lambda_j<\cdots \to \infty.$$ 
Let $d(\lambda_j)$ be the multiplicity of
the eigenvalue $\lambda_j$, and let $\{\varphi_{\lambda_j,\ell}(y)\}_{1\leq
\ell\leq d(\lambda_j)}$ be the eigenfunctions of $\widetilde{\Delta}_h$, that is
\begin{equation}
(\Delta_h+V_0(y))\varphi_{\lambda_j,\ell}(y)=\lambda_j \varphi_{\lambda_j,\ell}(y).
\end{equation}
This of course implies that
\begin{equation*}
(\Delta_h+V_0(y)+(n-2)^2/4)\varphi_{\lambda_j,\ell}(y)=(\lambda_j+(n-2)^2/4) \varphi_{\lambda_j,\ell}(y).
\end{equation*}
By the assumption that the operator $\Delta_h+V_0(y)+(n-2)^2/4$ is strictly positive, we have that $\nu_j^2:=\lambda_j+(n-2)^2/4>0$.  From now on we will drop the subscript $j$ in order to keep the notation simple, and 
define the set $\chi_\infty$ to be
\begin{equation}\label{set1}
\chi_\infty=\Big\{\nu: \nu=\sqrt{(n-2)^2/4+\lambda};~
\lambda~\text{is eigenvalue
of}~ \Delta_h+V_0(y)\Big\}.
\end{equation}
For $\nu\in\chi_\infty$, let $d(\nu)$ be the multiplicity of the eigenvalue $\nu^2$ of 
$\widetilde{\Delta}_h+(n-2)^2/4$ and let $\{\varphi_{\nu,\ell}(y)\}_{1\leq
\ell\leq d(\nu)}$ be the eigenfunctions of $\widetilde{\Delta}_h+(n-2)^2/4$. Then $d(\nu)=d(\lambda)$ and $\{\varphi_{\nu,\ell}(y)\}_{1\leq
\ell\leq d(\nu)}=\{\varphi_{\lambda,\ell}(y)\}_{1\leq
\ell\leq d(\lambda)}$.
Let $\mathcal{H}^{\nu}=\text{span}\{\varphi_{\nu,1},\ldots,
\varphi_{\nu,d(\nu)}\}$; we decompose $L^2(Y)$ as
\begin{equation*}
L^2(Y)=\bigoplus_{\nu\in\chi_\infty} \mathcal{H}^{\nu}.
\end{equation*} For $f\in L^2(X)$, we define the orthogonal projection $\pi_{\nu}:L^2(X)\to \mathcal{H}^{\nu}$
\begin{equation*}
\pi_{\nu}f= \sum_{\ell=1}^{d(\nu)}\varphi_{\nu,\ell}(y) a_{\nu,\ell}(r), \quad a_{\nu,\ell}(r)=\int_{Y}f(r,y)
\varphi_{\nu,\ell}(y) dh,
\end{equation*}
then we can write $f$ by separating variables as
\begin{equation}\label{sep.v}
f(z)=\sum_{\nu\in\chi_\infty}\pi_{\nu}f
=\sum_{\nu\in\chi_\infty}\sum_{\ell=1}^{d(\nu)}a_{\nu,\ell}(r)\varphi_{\nu,\ell}(y)
\end{equation}
so that
\begin{equation}\label{norm1}
\|f(z)\|^2_{L^2(Y)}=\sum_{\nu\in\chi_\infty}\sum_{\ell=1}^{d(\nu)}|a_{\nu,\ell}(r)|^2.
\end{equation}
Note that as the Riemannian metric $h$ on $Y$ is independent on $r$,  we can use the separation of variable method \cite{CT} to write $\mathcal{L}_V$  with respect to the coordinates $(r,y)$ as
\begin{equation}\label{operator-t}
\mathcal{L}_V=-\partial^2_r-\frac{n-1}r\partial_r+\frac1{r^2}\big(\Delta_h+V_0(y)\big).
\end{equation} 
Then, on each space $\mathcal{H}^{\nu}$, the action of the operator is given by
\begin{equation}\label{2.8}
\begin{split}
\LL_V=A_{\nu}:=-\partial_r^2-\frac{n-1}r\partial_r+r^{-2}\big(\nu^2-\frac14{(n-2)^2}\big).
\end{split}
\end{equation}
A crucial role in our argument will be played by the well known Hankel transform of order $\nu$: for $f\in L^2(X)$, we set
\begin{equation}\label{hankel}
(\mathcal{H}_{\nu}f)(\rho,y)=\int_0^\infty(r\rho)^{-\frac{n-2}2}J_{\nu}(r\rho)f(r,y)r^{n-1}dr,
\end{equation}
where the Bessel function of order $\nu$ is given by
\begin{equation}\label{Bessel}
J_{\nu}(r)=\frac{(r/2)^{\nu}}{\Gamma\left(\nu+\frac12\right)\Gamma(1/2)}\int_{-1}^{1}e^{isr}(1-s^2)^{(2\nu-1)/2} ds, \quad \nu>-1/2, r>0.
\end{equation}

We now recall the following very useful inequalities for the Bessel function (we refer for them to \cite{watson, MZZ})

\begin{lemma}\label{lem: J} Let $J_\nu(r)$ be the Bessel function defined in \eqref{Bessel}, then there exists a constant $C$ independent of $\nu$ such that

\begin{equation}\label{bessel-r}
|J_\nu(r)|\leq
\frac{Cr^\nu}{2^\nu\Gamma(\nu+\frac12)\Gamma(1/2)}\left(1+\frac1{\nu+1/2}\right).
\end{equation}
Moreover, if $R\gg1$, there exists a constant $C$ independent of $\nu$ and $R$ such that
\begin{equation}\label{est:b}
\int_{R}^{2R} |J_\nu(r)|^2 dr \leq C.
\end{equation}
\end{lemma}

For well-behaved functions $F$, by  \cite[(8.45)]{Taylor}, we have 
\begin{equation}\label{funct}
F(\mathcal{L}_V) g(r,y)=\sum_{\nu\in\chi_\infty}\sum_{\ell=1}^{d(\nu)} \varphi_{\nu,\ell}(y) \int_0^\infty F(\rho^2) (r\rho)^{-\frac{n-2}2}J_\nu(r\rho)b_{\nu,\ell}(\rho)\rho^{n-1} d\rho
\end{equation}
where $b_{\nu,\ell}(\rho)=(\mathcal{H}_{\nu}a_{\nu,\ell})(\rho)$ and $g(r,y)=\sum\limits_{\nu\in\chi_\infty}\sum\limits_{\ell=1}^{d(\nu)}a_{\nu,\ell}(r)~\varphi_{\nu,\ell}(y)$.

\subsection{Local smoothing estimates for Klein-Gordon}

We are now ready to state our result on the local-smoothing estimates.

\begin{proposition}\label{prop:loc} Let $u$ be a solution of \eqref{equ:KG} with $F=0$ and let $\nu_0$ be the positive square root of the smallest eigenvalue of $\Delta_h+V_0(y)+(n-2)^2/4$. Let $\varphi_0$ be given by \eqref{dp}. Then there exists a constant $C$
independent of $(u_0,u_1)$ such that:

$\bullet$ (Low frequency estimate) for $\beta\in [1, 1+\nu_0)$
\begin{equation}\label{l-local-s}
\begin{split}
\|r^{-\beta}\varphi_0(\sqrt{\LL_V}) u(t,z)\|_{L^2_t(\R;L^2(X))}\leq
C\left(\|u_0\|_{L^{2}(X)}+\|u_1\|_{L^{2}(X)}\right),
\end{split}
\end{equation}

$\bullet$ (High frequency estimate) for $\beta\in (1/2, 1+\nu_0)$
\begin{equation}\label{h-local-s}
\begin{split}
\|r^{-\beta}(1-\varphi_0)(\sqrt{\LL_V}) u(t,z)\|_{L^2_t(\R;L^2(X))}\leq
C\left(\|u_0\|_{H^{\beta-\frac12}(X)}+\|u_1\|_{H^{\beta-\frac32}(X)}\right),
\end{split}
\end{equation}
where $z=(r,y)\in X$.
\end{proposition}
\begin{remark} The two estimates are quite natural as indeed, as mentioned, the solutions to the Klein-Gordon equation behave like the Schr\"odinger ones for low frequencies and like the wave ones for high frequencies. It is known from \cite{ZZ1,ZZ2} that, for $\beta\in (1/2, 1+\nu_0)$, for the Schr\"odinger equation the following estimate holds
\begin{equation*}
\|r^{-\beta}e^{it\LL_V}f\|_{L^2_t(\R;L^2(X))}\leq C\|f\|_{\dot H^{\beta-1}(X)}, 
\end{equation*}
while for the wave one we have
\begin{equation*}
 \|r^{-\beta}e^{it\sqrt{\LL_V}}f\|_{L^2_t(\R;L^2(X))}\leq C\|f\|_{\dot H^{\beta-\frac12}(X)}.
\end{equation*}
\end{remark}

\begin{remark}\label{rem:loc-s} The restriction from above $\beta<1+\nu_0$ can become more specific in the following cases: 

(i) If $V=0$, then $\nu_0>(n-2)/2$, hence $\beta<\frac n2$;

(ii) If the initial data satisfy the condition that the projection $\pi_{\nu}(u_0,u_1)$ vanishes when $\nu\leq k$, then the above results hold for $\beta<1+k$. 

\end{remark}

\begin{proof} We first prove \eqref{l-local-s}. Let $u^l(t,z)=\varphi_0(\sqrt{\LL_V})u(t,z)$, then we can write
\begin{equation}\label{sleq2}
\begin{split}
u^l(t,z)&=\cos(t\sqrt{1+\LL_V})u^l_0+\frac{\sin(t\sqrt{1+\LL_V})}{\sqrt{1+\LL_V}}u^l_1
\\&=\frac12\left(e^{it\sqrt{1+\LL_V}}+e^{-it\sqrt{1+\LL_V}}\right)u^l_0+\frac1{2i}\frac{\left(e^{it\sqrt{1+\LL_V}}-e^{-it\sqrt{1+\LL_V}}\right)}{\sqrt{1+\LL_V}}u^l_1
\end{split}
\end{equation}
where $u^l_0=\varphi_0(\sqrt{\LL_V})u_0$ and $u^l_1=\varphi_0(\sqrt{\LL_V})u_1$ (we are using the apex $l$ to recall that we are dealing with low frequencies).
We only consider the contributions of $e^{it\sqrt{1+\LL_V}}u^l_0$ and $\frac{e^{it\sqrt{1+\LL_V}}}{\sqrt{1+\LL_V}}u^l_1$ since the others follow from the same argument. 

We write the solution using the harmonic expansion. By \eqref{sep.v}, we write  
\begin{equation*}
\begin{split}
u_0(z)=\sum_{\nu\in\chi_\infty}\sum_{\ell=1}^{d(\nu)}a_{\nu,\ell}(r)\varphi_{\nu,\ell}(y), \quad  b_{\nu,\ell}(\rho)=(\mathcal{H}_{\nu}a_{\nu,\ell})(\rho),\\
u_1(z)=\sum_{\nu\in\chi_\infty}\sum_{\ell=1}^{d(\nu)}\tilde{a}_{\nu,\ell}(r)\varphi_{\nu,\ell}(y), \quad  \tilde{b}_{\nu,\ell}(\rho)=(\mathcal{H}_{\nu}\tilde{a}_{\nu,\ell})(\rho).
\end{split}
\end{equation*}
By using \eqref{funct} with $F(\rho)=e^{it\sqrt{1+\rho^2}}\varphi_0(\rho)$, we have
\begin{equation*}
\begin{split}
e^{it\sqrt{1+\LL_V}}u^l_0&=\sum_{\nu\in\chi_\infty}\sum_{\ell=1}^{d(\nu)}\varphi_{\nu,\ell}(y)\int_0^\infty(r\rho)^{-\frac{n-2}2}J_{\nu}(r\rho)e^{
it\sqrt{1+\rho^2}}\varphi_0(\rho) b_{\nu,\ell}(\rho)\rho^{n-1}d\rho
\end{split}
\end{equation*}
and we apply \eqref{funct} with $F(\rho)=e^{it\sqrt{1+\rho^2}}\varphi_0(\rho)(1+\rho^2)^{-1/2}$ to obtain
\begin{equation*}
\begin{split}
\frac{e^{it\sqrt{1+\LL_V}}}{\sqrt{1+\LL_V}}u^l_1&=\sum_{\nu\in\chi_\infty}\sum_{\ell=1}^{d(\nu)}\varphi_{\nu,\ell}(y)\int_0^\infty(r\rho)^{-\frac{n-2}2}J_{\nu}(r\rho)\frac{e^{
it\sqrt{1+\rho^2}}}{\sqrt{1+\rho^2}}\varphi_0(\rho)\tilde{b}_{\nu,\ell}(\rho)\rho^{n-1}d\rho.
\end{split}
\end{equation*}
We first consider $e^{it\sqrt{1+\LL_V}}u^l_0$. 
Using the orthogonality, one has
\begin{equation*}
\begin{split}
\int_{Y} \big|\sum_{\nu\in\chi_\infty}\sum_{\ell=1}^{d(\nu)}\varphi_{\nu,\ell}(y)J_{\nu}(r\rho)b_{\nu,\ell}(\rho)
\big|^2 dy=\sum_{\nu\in\chi_\infty}\sum_{\ell=1}^{d(\nu)}\big|J_{\nu}(r\rho)b_{\nu,\ell}(\rho).
\big|^2
\end{split}
\end{equation*}
Hence by using the Plancherel theorem w.r.t.  time $t$, we obtain 
\begin{equation*}
\begin{split}
&\|r^{-\beta}e^{it\sqrt{1+\LL_V}}u^l_0\|^2_{L^2_t(\R;L^2(X))}\\
&=\sum_{\nu\in\chi_\infty}\sum_{\ell=1}^{d(\nu)}\int_0^\infty\int_0^\infty\big|(r\rho)^{-\frac{n-2}2}J_{\nu}(r\rho)\varphi_0(\rho) b_{\nu,\ell}(\rho)
\rho^{n-1}\big|^2 \frac{\sqrt{1+\rho^2}}\rho d\rho\, r^{n-1-2\beta}dr.
\end{split}
\end{equation*}
We recall that $\varphi_0(\rho)=\sum_{j\leq0}\varphi(2^{-j}\rho)$. We use the decomposition to obtain 
\begin{equation}\label{scal-reduce}
\begin{split}
&\|r^{-\beta}e^{it\sqrt{1+\LL_V}}u^l_0\|^2_{L^2_t(\R;L^2(X))}\\
&\lesssim \sum_{\nu\in\chi_\infty}\sum_{\ell=1}^{d(\nu)}\sum_{j\leq 0}\int_0^\infty\int_0^\infty\big|(r\rho)^{-\frac{n-2}2}J_{\nu}(r\rho)b_{\nu,\ell}(\rho)
\rho^{n-1}\big|^2 \varphi^2(2^{-j}\rho) \frac{d\rho}{\rho} r^{n-1-2\beta}dr\\&\lesssim
\sum_{\nu\in\chi_\infty}\sum_{\ell=1}^{d(\nu)}\sum_{j\leq0}\sum_{R\in2^{\Z}}2^{j(n-2+2\beta)}R^{n-1-2\beta}G_{\nu,\ell}(R,2^j),
\end{split}
\end{equation}
where
\begin{equation}\label{def:G}
\begin{split}
G_{\nu,\ell}(R,2^j)=\int_{R}^{2R}\int_{0}^\infty\big|(r\rho)^{-\frac{n-2}2}J_{\nu}(r\rho)b_{\nu,\ell}(2^j\rho)
\big|^2 \varphi^2(\rho)d\rho  dr.\end{split}
\end{equation}
Then we have the following lemma.
\begin{lemma} Let $G_{\nu,\ell}(R,2^j)$ be as in \eqref{def:G}, then
\begin{equation}\label{est:G}
G_{\nu,\ell}(R,2^j) \lesssim
\begin{cases}
R^{2\nu-n+3}2^{-nj}\|b_{\nu,\ell}(\rho)\varphi(2^{-j}\rho)\rho^{\frac{n-1}2}\|^2_{L^2},~
R\lesssim  1\\
R^{-(n-2)}2^{-nj}\|b_{\nu,\ell}(\rho)\varphi(2^{-j}\rho)\rho^{\frac{n-1}2}\|^2_{L^2},~R\gg1.
\end{cases}
\end{equation}
\end{lemma}
\

\begin{proof}
This lemma is proved in \cite[Proposition 4.2]{ZZ1}. For convenience, we provide the sketch of the proof again. 
To prove \eqref{est:G}, we break it into two
cases. \vspace{0.2cm}

$\bullet$ Case 1: $R\lesssim1$. Since $\rho\sim1$, thus
$r\rho\lesssim1$. By 
\eqref{bessel-r}, we obtain
\begin{equation*}
\begin{split}
G_{\nu,\ell}(R,2^j)&\lesssim\int_{R}^{2R}\int_{0}^\infty\Big| \frac{
(r\rho)^{\nu}(r\rho)^{-\frac{n-2}2}}{2^{\nu}\Gamma(\nu+\frac12)\Gamma(\frac12)}b_{\nu,\ell}(2^j\rho)\varphi(\rho) \Big|^2 d\rho dr\\& \lesssim
R^{2\nu-n+3}2^{-nj}\|b_{\nu,\ell}(\rho)\varphi(2^{-j}\rho)\rho^{\frac{n-1}2}\|^2_{L^2}.
\end{split}
\end{equation*}

$\bullet$ Case 2: $R\gg1$. Since $\rho\sim1$, thus $r\rho\gg 1$. We
estimate by \eqref{est:b} in Lemma \ref{lem: J} 
\begin{equation*}
\begin{split}
G_{\nu,\ell}(R,2^j)&\lesssim
R^{-(n-2)}\int_{0}^\infty\big|b_{\nu,\ell}(2^j\rho)\varphi(\rho)\big|^2\int_{R}^{2R}\big|J_{\nu}(
r\rho)\big|^2  dr d\rho\\& \lesssim R^{-(n-2)}\int_{0}^\infty\big|b_{\nu,\ell}(2^j\rho)\varphi(\rho)\big|^2 d\rho\lesssim
R^{-(n-2)}2^{-nj}\|b_{\nu,\ell}(\rho)\varphi(2^{-j}\rho)\rho^{\frac{n-1}2}\|^2_{L^2}.
\end{split}
\end{equation*}
Thus we obtain \eqref{est:G}. 

\end{proof}
Therefore we apply \eqref{est:G} to \eqref{scal-reduce} to obtain
\begin{equation*}
\begin{split}
&\|r^{-\beta}e^{it\sqrt{1+\LL_V}}u^l_0\|^2_{L^2_t(\R;L^2(X))}\\&\lesssim
\sum_{\nu\in\chi_\infty}\sum_{\ell=1}^{d(\nu)}\sum_{j\leq0}\Big(\sum_{R\in2^{\Z}, R\lesssim 1}2^{j(n-2+2\beta)}R^{n-1-2\beta}R^{2\nu-n+3}2^{-nj}
\\&\quad\qquad+\sum_{R\in2^{\Z}, R\gg 1}2^{j(n-2+2\beta)}R^{n-1-2\beta}R^{-(n-2)}2^{-nj}\Big)\|b_{\nu,\ell}(\rho)\varphi(2^{-j}\rho)\rho^{\frac{n-1}2}\|^2_{L^2}
\\&\lesssim
\sum_{\nu\in\chi_\infty}\sum_{\ell=1}^{d(\nu)}\sum_{j\leq0}2^{2j(\beta-1)}\Big(\sum_{R\in2^{\Z}, R\lesssim 1}R^{2(1+\nu-\beta)}
+\sum_{R\in2^{\Z}, R\gg 1}R^{1-2\beta}\Big)\|b_{\nu,\ell}(\rho)\varphi(2^{-j}\rho)\rho^{\frac{n-1}2}\|^2_{L^2}.
\end{split}
\end{equation*}
Note that if $\frac12<\beta<1+\nu_0 $ the summation in $R$ converges; this yields
\begin{equation*}
\begin{split}
&\|r^{-\beta}e^{it\sqrt{1+\LL_V}}u^l_0\|^2_{L^2_t(\R;L^2(X))}\lesssim
\sum_{\nu\in\chi_\infty}\sum_{\ell=1}^{d(\nu)}\sum_{j\leq0}2^{2j(\beta-1)}\|b_{\nu,\ell}(\rho)\varphi(2^{-j}\rho)\rho^{\frac{n-1}2}\|^2_{L^2}.
\end{split}
\end{equation*} 
Note that we benefit nothing from the factor $(1+\LL_V)^{-1/2}$ when we are restricted to low frequencies. Following the same argument, we also show
\begin{equation*}
\begin{split}
&\Big\|r^{-\beta}\frac{e^{it\sqrt{1+\LL_V}}}{\sqrt{1+\LL_V}}u_1^l\Big\|^2_{L^2_t(\R;L^2(X))}\lesssim
\sum_{\nu\in\chi_\infty}\sum_{\ell=1}^{d(\nu)}\sum_{j\leq0}2^{2j(\beta-1)}\|b_{\nu,\ell}(\rho)\varphi(2^{-j}\rho)\rho^{\frac{n-1}2}\|^2_{L^2}.
\end{split}
\end{equation*} 
Hence we obtain \eqref{l-local-s}. 

Next we prove \eqref{h-local-s}. Following the same argument, we are reduced to estimate
\begin{equation}\label{scal-reduce'}
\begin{split}
&\|r^{-\beta}e^{it\sqrt{1+\LL_V}}u^h_0\|^2_{L^2_t(\R;L^2(X))}\\
&\lesssim \sum_{\nu\in\chi_\infty}\sum_{\ell=1}^{d(\nu)}\sum_{j\geq 0}\int_0^\infty\int_0^\infty\big|(r\rho)^{-\frac{n-2}2}J_{\nu}(r\rho)b_{\nu,\ell}(\rho)
\rho^{n-1}\big|^2 \varphi^2(2^{-j}\rho) d\rho r^{n-1-2\beta}dr\\&\lesssim
\sum_{\nu\in\chi_\infty}\sum_{\ell=1}^{d(\nu)}\sum_{j\geq0}\sum_{R\in2^{\Z}}2^{j(n-1+2\beta)}R^{n-1-2\beta}G_{\nu,\ell}(R,2^j).
\end{split}
\end{equation}
Therefore we apply \eqref{est:G} again to \eqref{scal-reduce'} to obtain
\begin{equation*}
\begin{split}
&\|r^{-\beta}e^{it\sqrt{1+\LL_V}}u^h_0\|^2_{L^2_t(\R;L^2(X))}\\&\lesssim
\sum_{\nu\in\chi_\infty}\sum_{\ell=1}^{d(\nu)}\sum_{j\geq0}2^{2j(\beta-\frac12)}\Big(\sum_{R\in2^{\Z}, R\lesssim 1}R^{2(1+\nu-\beta)}
+\sum_{R\in2^{\Z}, R\gg 1}R^{1-2\beta}\Big)\|b_{\nu,\ell}(\rho)\varphi(2^{-j}\rho)\rho^{\frac{n-1}2}\|^2_{L^2}.
\end{split}
\end{equation*}
Note that if $\frac12<\beta<1+\nu_0 $ the summation in $R$ converges, thus we obtain
\begin{equation*}
\begin{split}
&\|r^{-\beta}e^{it\sqrt{1+\LL_V}}u^h_0\|^2_{L^2_t(\R;L^2(X))}\lesssim
\sum_{\nu\in\chi_\infty}\sum_{\ell=1}^{d(\nu)}\sum_{j\leq0}2^{2j(\beta-\frac12)}\|b_{\nu,\ell}(\rho)\varphi(2^{-j}\rho)\rho^{\frac{n-1}2}\|^2_{L^2}.
\end{split}
\end{equation*} 
Note that we gain a $2^{-j}$ from the factor $(1+\LL_V)^{-1/2}$ when we restrict to high frequencies. Following the same argument, we also have
\begin{equation*}
\begin{split}
&\Big\|r^{-\beta}\frac{e^{it\sqrt{1+\LL_V}}}{\sqrt{1+\LL_V}}u_1^h\Big\|^2_{L^2_t(\R;L^2(X))}\lesssim
\sum_{\nu\in\chi_\infty}\sum_{\ell=1}^{d(\nu)}\sum_{j\leq0}2^{2j(\beta-\frac32)}\|b_{\nu,\ell}(\rho)\varphi(2^{-j}\rho)\rho^{\frac{n-1}2}\|^2_{L^2}.
\end{split}
\end{equation*} 
Hence \eqref{h-local-s} follows.

\end{proof}

\begin{corollary}\label{cor:loc} Let $u$ be the solution of \eqref{equ:KG} with $F=0$, then there exists a constant $C$
independent of $(u_0,u_1)$ such that
\begin{equation}\label{cor:local-s}
\begin{split}
\|r^{-\beta}u(t,z)\|_{L^2_t(\R;L^2(X))}\leq
C\left(\|u_0\|_{H^{\beta-\frac12}(X)}+\|u_1\|_{H^{\beta-\frac32}(X)}\right),
\end{split}
\end{equation}
where $z=(r,y)\in X$, $1\leq \beta<1+\nu_0$ with $\nu_0>0$ such that $\nu_0^2$ is the smallest eigenvalue of $\Delta_h+V_0(y)+(n-2)^2/4$.
\end{corollary}

\begin{proof}
This is a consequence of \eqref{l-local-s} and \eqref{h-local-s}.

\end{proof}

\section{The proof of Theorem \ref{thm:Strichartz} }\label{sec:7}

In this section, taking the influence of the potential $V$ into account, by making use of  Propositions \ref{Str-L0}, \ref{prop:inh} and \ref{prop:loc} we prove the Strichartz estimates in Theorem \ref{thm:Strichartz}.

\subsection {The Strichartz estimates \eqref{V-stri} } In this subsection, we mainly prove the homogeneous Strichartz estimates \eqref{V-stri} for all $q>2$ and $(q,r)\in \Lambda_{s,\theta,\nu_0}$. 
Due to the dependence on $\nu_0$ of the set $\Lambda_{s,\theta,\nu_0}$, we consider two cases.\vspace{0.2cm}

{\bf Case 1:} $0<\nu_0\leq \frac{1-\theta}{n-1+\theta}$ with $0\leq\theta\leq 1$.
 In this case, 
the set $\Lambda_{s,\theta,\nu_0}$ is the region $ACE$ which is contained in $ABF$ ($n\geq4;\; {\rm or}\, n=3, 0<\theta\leq 1$) or $ABO$ ($n=3, \theta=0$) respectively. 
\begin{center}
 \begin{tikzpicture}[scale=1]
\draw[->] (0,0) -- (4,0) node[anchor=north] {$\frac{1}{q}$};
\draw[->] (0,0) -- (0,4)  node[anchor=east] {$\frac{1}{r}$};
\draw (0,0) node[anchor=north] {$O$}
(3,0) node[anchor=north] {$\frac12$};
\draw  (0, 3) node[anchor=east] {$\frac12$}
       (0, 1.8) node[anchor=east] {$\frac12-\frac{1+\nu_0}{n}$}
       (0, 1.2) node[anchor=east] {$\frac12-\frac{1}{n-1+\theta}$};

\draw[thick] (3,0) -- (3,1.2)  
              (3,1.2) -- (0,3);

\draw[dashed,thick] (0,1.2) -- (3,1.2);
 \draw[red, dashed,thick]         (2.0,1.8) -- (0,1.8); 
\filldraw[fill=gray!50](0,1.8)--(2,1.8)--(0,3); 
\draw (-0.1,3.2) node[anchor=west] {$A$};
\draw (2.9,1.2) node[anchor=west] {$B$};
\draw (1.8,1.6) node[anchor=west] {$C$};
\draw (1.7,2.0) node[anchor=west] {$C'$};
\draw (2.9,0.15) node[anchor=west] {$D$};
\draw (-0.1,1.6) node[anchor=west] {$E$};
\draw (-0.1,1.0) node[anchor=west] {$F$};
\draw (1.65,2.88) node[anchor=west] {$\frac2q+\frac{n-1+\theta}{r}=\frac{n-1+\theta}{2}$};

\draw[<-] (1.6,2.1) -- (2,2.6) node[anchor=south]{$~$};

\path (2,-1) node(caption){Fig 1. $n\geq4, 0\leq\theta\leq 1; \;{\rm or}\,  n=3, 0<\theta\leq 1$};  

\draw (1.75,1.95) circle (0.06);

\draw[->] (8,0) -- (12,0) node[anchor=north] {$\frac{1}{q}$};
\draw[->] (8,0) -- (8,4)  node[anchor=east] {$\frac{1}{r}$};
\path (9.6,-1) node(caption){Fig 2. $n=3, \theta=0$};  

\draw  (8.1, -0.1) node[anchor=east] {$O$};
\draw  (11, 0) node[anchor=north] {$\frac12$};
\draw  (8, 3) node[anchor=east] {$\frac12$};

\draw[thick] (8,3) -- (11,0);  
\draw (7.9,3.15) node[anchor=west] {$A$};
\draw (10.9,0.2) node[anchor=west] {$B$};
\draw (10,2.6) node[anchor=west] {$\frac{2}{q}+\frac{2}{r}=1$};

\draw (10.4,0.6) circle (0.06);

\filldraw[fill=gray!50](8,3)--(8,0.6)--(10.4,0.6); 

\draw (7.9,0.45) node[anchor=west] {E};
\draw (9.9,0.95) node[anchor=west] {$C'$};
\draw (10,0.45) node[anchor=west] {C};
\draw   (8, 0.6) node[anchor=east] {$\frac12-\frac{1+\nu_0}{n}$};
 \draw[red, dashed,thick]         (8.0,0.6) -- (10.5,0.6); 

\draw[<-] (9,2.1) -- (10,2.6) node[anchor=south]{$~$};

\path (6,-1.5) node(caption){Diagrammatic picture of the range of $(q,r)$, when $0<\nu_0\leq (1-\theta)/(n-1+\theta)$.};  

\end{tikzpicture}

\end{center}

If $(q,r)\in\Lambda_{s,\theta,\nu_0}$, on the one hand, by $(q,r)\in \Lambda_{s,\theta}$, we have
\begin{equation}\label{s:index}
s=(n+\theta)\Big(\frac12-\frac1r\Big)-\frac1q\geq\frac12(n+1+\theta)\Big(\frac12-\frac1r\Big)\geq0
\end{equation}
and on the other hand, by $0<\nu_0\leq \frac{1-\theta}{n-1+\theta}$, we have 
\begin{equation}\label{s:index'}
s=(n+\theta)\Big(\frac12-\frac1r\Big)-\frac1q\leq(n+\theta)\frac{1+\nu_0}n-\frac1q\leq 1.
\end{equation}

Therefore, without loss of generality,  we may assume
$0\leq s\leq 1$. 

Our strategy is then the following:

$\bullet$ prove \eqref{V-stri} at the point $A$ and $C'$, where $C'$ is on the line $AC$ and is close to $C$,

$\bullet$ obtain the Strichartz estimates  on the line $AC$ (except point $C$) by interpolation,

$\bullet$ show the inequalities in the  region $ACE$ (except $CE$) by Sobolev inequalities.\vspace{0.2cm}

It is easy to prove the Strichartz estimate when $(q,r)=(\infty, 2)\in\Lambda_{0,\theta}$ (i.e. $A$ in the
above figures) from the spectral theory on $L^2$.\vspace{0.2cm}

Consider now any fixed point $C'=(\frac1q,\frac1r)$ such that $\frac2q=(n-1+\theta)(\frac12-\frac1r)$ and $\frac1r=\frac12-\frac{1+\nu_0}n+\epsilon$ with $0<\epsilon\ll1$. In other words, the point $C'$
is on the line $AC$ and is close to $C$, and let us prove \eqref{V-stri} at the point $C'$. Recall that as $s=(n+\theta)\Big(\frac12-\frac1r\Big)-\frac1q$, then
\begin{equation*}
\frac12\leq s=\big(\frac{1+\nu_0}n-\epsilon\big)\frac{n+1+\theta}2\leq \frac{n+1+\theta}{2(n-1+\theta)}-\frac{\epsilon}2(n+1+\theta).
\end{equation*}
We thus consider the free Klein-Gordon equation
\begin{equation*}
\partial_{t}^2u+\LL_V u+u=0, \quad u(0)=u_0,
~\partial_tu(0)=u_1,
\end{equation*}
and without losing in generality, we can assume $u_1=0$.
In order to prove  \eqref{V-stri} at the point $C'$, we split the initial data into two parts,  $u_0=u_{0,l}+u_{0,h}$ where $u_{0,h}=u_0-u_{0,l}$ and
 \begin{equation*}
 u_{0,l}=\sum_{\nu\in A}\sum_{\ell=1}^{d(\nu)}a_{\nu,\ell}(r)\varphi_{\nu,\ell}(y), \quad A=\{\nu\in\chi_\infty: \nu\leq 1+\nu_0\}.
 \end{equation*}
 We remark here that the notation $u_{0,l}$ has nothing to do with $u_0^l=\varphi_0(\sqrt{\LL_V})u_0$ defined before. Correspondingly, we split the solution into two parts, $u=u_l+u_h$, where $u_l$ and $u_h$ satisfy 
\begin{equation}\label{KG-l}
\partial_{t}^2u_l+\LL_V u_l+u_l=0, \quad u(0)=u_{0,l},
~\partial_tu(0)=0,
\end{equation}
and
\begin{equation}\label{KG-h}
\partial_{t}^2u_h+\LL_V u_h+u_h=0, \quad u(0)=u_{0,h},
~\partial_tu(0)=0.
\end{equation}
Note that $\pi_{\nu}u_{0,h}=0$ when $\nu\leq k:=1+\nu_0$, hence one also has $\pi_{\nu}u_{h}=0$ when $\nu\leq k:=1+\nu_0$. 

We first consider $u_h(t,z)$. By the Duhamel formula,
we have 
\begin{equation}\label{h-duhamel'}
\begin{split}
u_h(t,z)&=\frac{e^{it\sqrt{1+\LL_V}}+e^{-it\sqrt{1+\LL_V}}}2 u_{0,h}\\&=\frac{e^{it\sqrt{1+\Delta_g}}+e^{-it\sqrt{1+\Delta_g}}}2 u_{0,h}+\int_0^t\frac{\sin{(t-\tau)\sqrt{1+\Delta_g}}}
{\sqrt{1+\Delta_g}}(V(z)u_h(\tau,z))d\tau.
\end{split}
\end{equation}
By using \eqref{h-duhamel'} and Proposition \ref{Str-L0}, we have
\begin{equation*}
\begin{split}
&\|u_h(t,z)\|_{L^q(\R;L^{r}(X))}\\&\lesssim \|u_{0,h}\|_{ H^{s}(X)}+\Big\|\int_0^t\frac{\sin{(t-\tau)\sqrt{1+\Delta_g}}}
{\sqrt{1+\Delta_g}}(V(z)u_h(\tau,z))d\tau\Big\|_{_{L^q(\R;L^{r}(X))}}.
\end{split}
\end{equation*}
Notice that since $0\leq s\leq 1$, as mentioned in Remark \ref{rem:equiv} the two Sobolev norms are equivalent (independently from the fact that $V=0$); therefore we can safely use the notation $H^s$ without any confusion.
Now our main task is to prove
\begin{equation}\label{est:h-inh}
\begin{split}
\Big\|\int_0^t\frac{\sin{(t-\tau)\sqrt{1+\Delta_g}}}
{\sqrt{1+\Delta_g}}(V(z)u_h(\tau,z))d\tau\Big\|_{_{L^q(\R;L^{r}(X))}}\lesssim \|u_0\|_{ H^{s}(X)}.
\end{split}
\end{equation}
To this end, we estimate
\begin{equation}\label{est:h-inh'}
\begin{split}
&\Big\|\int_0^t\frac{\sin{(t-\tau)\sqrt{1+\Delta_g}}}
{\sqrt{1+\Delta_g}}(V(z)u_h(\tau,z))d\tau\Big\|_{_{L^q(\R;L^{r}(X))}}
\\&\lesssim \Big\|\int_0^t\frac{\sin{(t-\tau)\sqrt{1+\Delta_g}}}
{\sqrt{1+\Delta_g}}(1-\varphi_0)(\sqrt{\Delta_g})(V(z)u_h(\tau,z))d\tau\Big\|_{_{L^q(\R;L^{r}(X))}}\\&+\Big\|\int_0^t\frac{\sin{(t-\tau)\sqrt{1+\Delta_g}}}
{\sqrt{1+\Delta_g}}\varphi_0(\sqrt{\Delta_g})(V(z)u_h(\tau,z))d\tau\Big\|_{_{L^q(\R;L^{r}(X))}}.
\end{split}
\end{equation}
Hence \eqref{est:h-inh} is the consequence of the following lemma.
\begin{lemma}
We have
\begin{equation}\label{est:h-h-inh}
\begin{split}
 \Big\|\int_0^t\frac{\sin{(t-\tau)\sqrt{1+\Delta_g}}}
{\sqrt{1+\Delta_g}}(1-\varphi_0)(\sqrt{\Delta_g})(V(z)u_h(\tau,z))d\tau\Big\|_{_{L^q(\R;L^{r}(X))}}\lesssim \|u_0\|_{ H^{s}(X)},
\end{split}
\end{equation}
and
\begin{equation}\label{est:h-l-inh}
\begin{split}
\Big\|\int_0^t\frac{\sin{(t-\tau)\sqrt{1+\Delta_g}}}
{\sqrt{1+\Delta_g}}\varphi_0(\sqrt{\Delta_g})(V(z)u_h(\tau,z))d\tau\Big\|_{_{L^q(\R;L^{r}(X))}}\lesssim \|u_0\|_{ H^{s}(X)}.
\end{split}
\end{equation}

\end{lemma}
  
\begin{proof}
We first prove \eqref{est:h-h-inh}. Let $\beta=\frac32-s$, then $\beta$ is close to $1-\frac1{n-1+\theta}$ from above, hence $\frac12<\beta<1+\nu_0$. We  define the operator 
$$T: L^2(X)\to L^2(\R;L^2(X)), \quad Tf= r^{-\beta}e^{it\sqrt{1+\Delta_g}}(1-\varphi_0)(\sqrt{\Delta_g})(1+ \Delta_g)^{\frac12(\frac12-\beta)} f.$$ 
Thus from the proof of Proposition \ref{prop:loc} when $V=0$,  it follows that $T$ is a bounded operator.
By duality, its adjoint $T^*$  
$$T^*: L^2(\R;L^2(X))\to L^2, \quad T^* F=\int_{\tau\in\R}(1+\Delta_g)^{\frac12(\frac12-\beta)} (1-\varphi_0)(\sqrt{\Delta_g}) e^{-i\tau\sqrt{1+\Delta_g}}  r^{-\beta}  F(\tau)d\tau$$ 
is also bounded. Define the operator 
$$B: L^2(\R;L^2(X))\to L^q(\R;L^r(X)), \quad B F=\int_{\tau\in\R} \frac{e^{i(t-\tau)\sqrt{1+\Delta_g}}}{\sqrt{1+\Delta_g}}(1-\varphi_0)(\sqrt{\Delta_g}) r^{-\beta}F(\tau)d\tau.$$
Hence by the Strichartz estimate \eqref{Str-L0-est}   with $s=\frac32-\beta$, one has
\begin{equation}\label{BF}
\begin{split}
&\|B F\|_{L^q(\R;L^r(X))}\\&=\big\| e^{i t\sqrt{1+\Delta_g}}\int_{\tau\in\R}\frac{e^{-i\tau\sqrt{1+\Delta_g}}}{\sqrt{1+\Delta_g}} (1-\varphi_0)(\sqrt{\Delta_g})r^{-\beta} F(\tau)d\tau\big\|_{L^q(\R;L^r(X))}\\
&\lesssim \big\|(1+\Delta_g)^{\frac12(\frac32-\beta)} \int_{\tau\in\R}\frac{e^{-i\tau\sqrt{1+\Delta_g}}}{\sqrt{1+\Delta_g}}(1-\varphi_0)(\sqrt{\Delta_g}) r^{-\beta} F(\tau)d\tau\big\|_{L^2(X)}\\&=\|T^*F\|_{L^2}\lesssim \|F\|_{L^2(\R;L^2(X))}.
\end{split}
\end{equation}
Now we are ready to prove inquality \eqref{est:h-h-inh}. As
$$\sin\big((t-\tau)\sqrt{1+\Delta_g}\big)=\frac{1}{2i}\big(e^{i(t-\tau)\sqrt{1+\Delta_g}}-e^{-i(t-\tau)\sqrt{1+\Delta_g}}\big),$$
 by \eqref{BF}, we have
\begin{equation*}
\begin{split}
&\Big\|\int_\R\frac{\sin{(t-\tau)\sqrt{1+\Delta_g}}}
{\sqrt{1+\Delta_g}}(1-\varphi_0)(\sqrt{\Delta_g}) (V(z)u_h(\tau,z))d\tau\Big\|_{L^q(\R;L^{r}(X))}\\&\lesssim \|B(r^{\beta}V(z)u_h(\tau,z))\|_{L^q(\R;L^r(X))}\lesssim \|r^{\beta-2}u_h(\tau,z))\|_{L^2(\R;L^2(X))}
\\&\lesssim \|u_{0,h}\|_{ H^{\frac32-\beta}(X)},
\end{split}
\end{equation*}
where we have used Corollary \ref{cor:loc} and Remark \ref{rem:loc-s} in the last inequality since $\beta=(1-\frac1{n-1+\theta})_+$ satisfies $1\leq 2-\beta<2+\nu_0$.
Since $q>2$ and  $s=\frac32-\beta$, by the Christ-Kiselev lemma \cite{CK}, we have \eqref{est:h-h-inh}.  \vspace{0.2cm}

We next prove \eqref{est:h-l-inh}, in a similar way. We define the operator 
$$T: L^2(X)\to L^2(\R;L^2(X)), \quad Tf= r^{-1}e^{it\sqrt{1+\Delta_g}}\varphi_0(\sqrt{\Delta_g}) f.$$ 
From the proof of Proposition \ref{prop:loc} again,  it follows that $T$ is a bounded operator.
By duality, its adjoint $T^*$  
$$T^*: L^2(\R;L^2(X))\to L^2, \quad T^* F=\int_{\tau\in\R}\varphi_0(\sqrt{\Delta_g}) e^{-i\tau\sqrt{1+\Delta_g}}  r^{-1}  F(\tau)d\tau$$ 
is also bounded. Define the operator 
$$B: L^2(\R;L^2(X))\to L^q(\R;L^r(X)), \quad B F=\int_{\tau\in\R} \frac{e^{i(t-\tau)\sqrt{1+\Delta_g}}}{\sqrt{1+\Delta_g}}\varphi_0(\sqrt{\Delta_g}) r^{-1}F(\tau)d\tau.$$
Hence by the Strichartz estimate \eqref{Str-L0-est} with $s=(n+\theta)\Big(\frac12-\frac1r\Big)-\frac1q$, one has
\begin{equation}\label{BF'}
\begin{split}
&\|B F\|_{L^q(\R;L^r(X))}\\&=\big\| e^{i t\sqrt{1+\Delta_g}}\int_{\tau\in\R}\frac{e^{-i\tau\sqrt{1+\Delta_g}}}{\sqrt{1+\Delta_g}} \varphi_0(\sqrt{\Delta_g})r^{-1} F(\tau)d\tau\big\|_{L^q(\R;L^r(X))}\\
&\lesssim \big\|(1+\Delta_g)^{\frac s2} \int_{\tau\in\R}\frac{e^{-i\tau\sqrt{1+\Delta_g}}}{\sqrt{1+\Delta_g}}\varphi_0(\sqrt{\Delta_g}) r^{-1} F(\tau)d\tau\big\|_{L^2(X)}\\&\lesssim\|T^*F\|_{L^2}\lesssim \|F\|_{L^2(\R;L^2(X))}.
\end{split}
\end{equation}
Now we estimate \eqref{est:h-l-inh}. 
By using \eqref{BF'}, Corollary \ref{cor:loc} and similar argument as above, we have 
\begin{equation*}
\begin{split}
&\Big\|\int_\R\frac{\sin{(t-\tau)\sqrt{1+\Delta_g}}}
{\sqrt{1+\Delta_g}}\varphi_0(\sqrt{\Delta_g}) (V(z)u_h(\tau,z))d\tau\Big\|_{L^q(\R;L^{r}(X))}\\&\lesssim \|B(r V(z)u_h(\tau,z))\|_{L^q(\R;L^r(X))}\lesssim \|r^{-1}u_h(\tau,z))\|_{L^2(\R;L^2(X))}
\\&\lesssim \|u_{0,h}\|_{ H^{\frac12}(X)}.
\end{split}
\end{equation*}
Since $1/2\leq s\leq 1$, we can replace the $H^{1/2}$-norm by $H^s$-norm. Due to $q>2$, by the Christ-Kiselev lemma \cite{CK} we obtain \eqref{est:h-l-inh}. 
\end{proof}
Therefore we have proved \eqref{est:h-inh}; hence, we obtain
\begin{equation*}
\begin{split}
\|u_h(t,z)\|_{L^q(\R;L^{r}(X))}\lesssim \|u_{0}\|_{ H^{s}(X)}.
\end{split}
\end{equation*}
Next we aim to prove the same inequality for $u_l$, that is
\begin{equation}\label{est:l-hom}
\begin{split}
\|u_l(t,z)\|_{L^q(\R;L^{r}(X))}\lesssim \|u_{0}\|_{ H^{s}(X)}.
\end{split}
\end{equation}
Notice that the above argument breaks down since $\pi_\nu(u_l)$ does not vanish when $\nu<k=1+\nu_0$: this fact will lead to a tighter restriction on $\beta$.
Nevertheless, if $\nu<k:=1+\nu_0$, we can follow the argument of \cite{PST} which treated the radial case. Since
\begin{equation*}
\begin{split}
u_l(t,z)&=\frac{e^{it\sqrt{1+\LL_V}}+e^{-it\sqrt{1+\LL_V}}}2 u_{0,l},\end{split}
\end{equation*}
we only consider the Strichartz estimate for $e^{it\sqrt{1+\LL_V}}u_{0,l}$. 
By using \eqref{funct}, we write
 \begin{equation*}
\begin{split} e^{it\sqrt{1+\LL_V}}u_{0,l}&=
\sum_{\nu\in A}\sum_{\ell=1}^{d(\nu)}\varphi_{\nu,\ell}(y)\int_0^\infty(r\rho)^{-\frac{n-2}2}J_{\nu}(r\rho)e^{it\sqrt{1+\rho^2}}\mathcal{H}_\nu(a_{\nu,\ell})\rho^{n-1}d\rho,\\
&=\sum_{\nu\in A}\sum_{\ell=1}^{d(\nu)}\varphi_{\nu,\ell}(y)\mathcal{H}_\nu[e^{it\sqrt{1+\rho^2}}\mathcal{H}_\nu(a_{\nu,\ell})](r),
\end{split}
\end{equation*}
where $\mathcal{H}_\nu$ is the Hankel transform defined in \eqref{hankel}.

By triangle inequality and the H\"ormander's $L^\infty$-estimate for eigenfunctions \cite{Hor} i.e. $\|\varphi_{\nu,\ell}\|_{L^\infty(Y)}\leq \nu^{\frac{n-1}2}$, one has
 \begin{equation}\label{est:l-hom'}
\begin{split} &\|e^{it\sqrt{1+\LL_V}}u_{0,l}\|_{L^q(\R;L^r(X))} \\&\leq C_{\nu_0}
\sum_{\nu\in A}\sum_{\ell=1}^{d(\nu)}\left\|\mathcal{H}_\nu[e^{it\sqrt{1+\rho^2}}\mathcal{H}_\nu(a_{\nu,\ell})](r)\right\|_{L^q(\R;L^r_{r^{n-1}dr})}.
\end{split}
\end{equation}
For our purpose, we need the properties on the Hankel transforms which are proved in \cite[Corollary 3.2, Theorem 3.8]{PST} .

\begin{lemma}\label{K0} Let $\mathcal{H}_\nu$ be the Hankel transform of order $\nu$ as defined in \eqref{hankel}, and let $\mathcal{K}^0_{\mu,\nu}:=\mathcal{H}_\mu\mathcal{H}_\nu$ where $\mu=(n-2)/2$. Then

$(\mathrm i)$ $\mathcal{H}_\mu\mathcal{H}_\mu=Id$, 

$(\mathrm {ii})$  the operator $\mathcal{K}^0_{\nu,\mu}$
is bounded on $L^p_{r^{n-1}dr}([0,\infty))$ if
$$\max\{((n-2)/2-\nu)/n,0\}<1/p<1,$$

$(\mathrm {iii})$  the operator $\mathcal{K}^0_{\mu,\nu}$ is continuous on $H^s$ provided
$$ -\min\{\mu,\nu,\mu-s\}< 1<2+\min\{\mu,\nu,\nu-s\}.$$

\end{lemma} 
By using this lemma, the operator $\mathcal{K}^0_{\nu,\mu}$ is bounded in 
$L^r_{r^{n-1}dr}([0,\infty))$ under the assumption $\frac1r>\frac12-\frac{1+\nu_0}n$ since $\nu\geq \nu_0$. Therefore from \eqref{est:l-hom'}, we obtain
 \begin{equation}\label{est:l-hom''}
\begin{split} 
&\|e^{it\sqrt{1+\LL_V}}u_{0,l}\|_{L^q(\R;L^r(X))}\\&\leq C_{\nu_0}\sum_{\nu\in A}\sum_{\ell=1}^{d(\nu)}\left\|(\mathcal{H}_\nu\mathcal{H}_\mu)\mathcal{H}_\mu[e^{it\sqrt{1+\rho^2}}\mathcal{H}_\mu(\mathcal{H}_\mu\mathcal{H}_\nu)(a_{\nu,\ell})](r)\right\|_{L^q(\R;L^r_{r^{n-1}dr})}\\&\leq C_{\nu_0}\sum_{\nu\in A}\sum_{\ell=1}^{d(\nu)}\left\|\mathcal{H}_\mu[e^{it\sqrt{1+\rho^2}}\mathcal{H}_\mu \mathcal{K}^0_{\mu,\nu}(a_{\nu,\ell})](r)\right\|_{L^q(\R;L^r_{r^{n-1}dr})}.\end{split}
\end{equation}
On the other hand, the propagator 
\begin{equation*}
\mathcal{H}_\mu e^{it\sqrt{1+\rho^2}}\mathcal{H}_\mu
\end{equation*}
 is the same as the classical Klein-Gordon propagator in the radial case in which the Strichartz estimates hold. 
From \eqref{est:l-hom''} and using (iii) in Lemma \ref{K0} with $1/2\leq s\leq 1$, we thus get 
 \begin{equation*}
\begin{split} 
\|e^{it\sqrt{1+\LL_V}}u_{0,l}\|_{L^q(\R;L^r(X))}
&\leq C_{\nu_0}\sum_{\nu\in A}\sum_{\ell=1}^{d(\nu)}\left\|\mathcal{K}^0_{\mu,\nu}(a_{\nu,\ell})](r)\right\|_{H^s} 
\\&\leq C_{\nu_0}\sum_{\nu\in A}\sum_{\ell=1}^{d(\nu)}\left\|a_{\nu,\ell}(r)\right\|_{H^s}.\end{split}
\end{equation*}
In the second inequality, we use \cite[Theorem 3.8]{PST}. 
Recall that from \eqref{set1} we have $\nu=\sqrt{(n-2)^2/4+\lambda}$ where $\lambda$ is an eigenvalue of the operator $\Delta_h+V_0(y)$. It is known that $\lambda$ is in a discrete 
set and moreover $\lambda_0<\lambda_1<\cdots<\lambda_j<\cdots \to \infty$. 
Therefore we have $\nu_j\to\infty$ as $j\to\infty$. As a consequence, by the definition of $A=\{\nu\in\chi_\infty: \nu\leq 1+\nu_0\}$, there exists a constant $C_{\nu_0}$ depending on $\nu_0$ such that the cardinality of the set $A$ is $\sharp A\leq C_{\nu_0}$. Thus
 \begin{equation*}
\begin{split} 
\|e^{it\sqrt{1+\LL_V}}u_{0,l}\|_{L^q(\R;L^r(X))}\leq \tilde{C}_{\nu_0}\left(\sum_{\nu\in A}\sum_{\ell=1}^{d(\nu)}\left\|a_{\nu,\ell}(r)\right\|^2_{ H^s}\right)^{1/2}\leq \tilde{C}_{\nu_0}\|u_{0,l}\|_{ H^s}
\end{split}
\end{equation*}
which implies \eqref{est:l-hom}, and this completes the proof of \eqref{V-stri} at the point $C'$.\\

By relying on standard interpolation theory, we then obtain Strichartz estimates on the line $AC'$. 

Now we show the Strichartz estimates in the region $ACE$. For this purpose, we need a Sobolev inequality associated to the operator $\LL_V$. 

\begin{proposition}[Sobolev inequality for $\LL_V$]\label{P:sobolev}  Let $n\geq 3$ and $\nu_0$ be defined as above. Suppose $0\leq \sigma<1+\nu_0$ and $2\leq p,q<\infty$.  Then
\begin{equation}\label{est:sobolev}
\big\|f(z)\big\|_{L^q(X)}\lesssim \big\|(1+\LL_V)^\frac{\sigma}2f\big\|_{L^p(X)}
\end{equation}
holds for $\sigma=\tfrac{n}p-\tfrac{n}q$ and
\begin{equation}\label{est:sobolev_hyp}
\frac{n}{\min\{1+\frac n2+\nu_0-\sigma, n\}}<q<\frac{n}{\max\{\frac n2-1-\nu_0, 0\}}.\end{equation}
\end{proposition}

\begin{proof} The proof follows from the argument of \cite[Proposition 3.1, Corollary 3.1, Remark 3.4]{ZZ2}. We omit the details here.

\end{proof}

For any point $(\frac1q,\frac1r)$ in the region $ACE$,  there exists a point $(\frac1q,\frac1{r_0})\in\Lambda_{s_0,\theta,\nu_0}$ on the line $AC'$.
Since $(q,r)\in\Lambda_{s,\theta,\nu_0}$, the condition $1/r>1/2-(1+\nu_0)/n$ guarantees \eqref{est:sobolev_hyp}. By using the Sobolev inequality in Proposition \ref{P:sobolev} ,
we obtain 
\begin{equation*}
\begin{split}
\|u(t,z)\|_{L^q(\R;L^{r}(X))}&\lesssim \|(1+\LL_V)^{\frac \sigma2} u(t,z)\|_{L^q(\R;L^{r_0}(X))}\\&\lesssim \|(1+\LL_V)^{\frac \sigma2}u_0\|_{ H^{s_0}(X)}+\|(1+\LL_V)^{\frac \sigma2}u_1\|_{ H^{s_0-1}(X)}
\\&\lesssim \|u_0\|_{ H^{s}(X)}+\|u_1\|_{ H^{s-1}(X)}, \quad s\geq s_0+\sigma, \end{split}
\end{equation*}
where $s=(n+\theta)(\frac12-\frac1r)-\frac1q$, $s_0=(n+\theta)(\frac12-\frac1{r_0})-\frac1q$ and $\sigma=n(\frac1 {r_0}-\frac1r)$. This concludes the proof in the region $ACE$.

 \vspace{0.2cm}

{\bf Case 2:} $\nu_0> \frac{1-\theta}{n-1+\theta}$ with $0\leq\theta\leq 1$. In this case, 
the set $\Lambda_{s,\theta,\nu_0}$ is the region $ABCE$ which  contains $ABF$ ($n\geq4; or\, n=3, 0<\theta\leq 1$) or $ABO$ ($n=3, \theta=0$) respectively. However the most relevant case, $q=2$, is excluded 
in \eqref{V-stri}. Since $q>2$, the argument of Case 1 proves the Strichartz estimates for all points of the region $ABCE$ (except the line $BC$).\vspace{0.2cm}

\begin{center}
 \begin{tikzpicture}[scale=1]
\draw[->] (0,0) -- (4,0) node[anchor=north] {$\frac{1}{q}$};
\draw[->] (0,0) -- (0,4)  node[anchor=east] {$\frac{1}{r}$};
\draw (0,0) node[anchor=north] {O}
(3,0) node[anchor=north] {$\frac12$};
\draw  (0, 3) node[anchor=east] {$\frac12$}
       (0, 0.5) node[anchor=east] {$\frac12-\frac{1+\nu_0}{n}$}
       (0, 1.2) node[anchor=east] {$\frac12-\frac{1}{n-1+\theta}$};

\draw[thick] (3,0) -- (3,1.2)  
              (3,1.2) -- (0,3);

\draw[red, dashed,thick] (0,1.2) -- (3,1.2);
 \draw[red, dashed,thick]         (3.0,0.5) -- (0,0.5); 
\filldraw[fill=gray!50](0,0.5)--(3,0.5)--(3,1.2)--(0,3); 
\draw (-0.1,3.2) node[anchor=west] {$A$};
\draw (2.9,1.2) node[anchor=west] {$B$};
\draw (3,0.5) node[anchor=west] {$C$};
\draw (2.9,0.15) node[anchor=west] {$D$};
\draw (-0.1,0.3) node[anchor=west] {$E$};
\draw (-0.1,1.0) node[anchor=west] {$F$};
\draw (1.65,2.88) node[anchor=west] {$\frac2q+\frac{n-1+\theta}{r}=\frac{n-1+\theta}{2}$};

\draw[<-] (1.6,2.1) -- (2,2.6) node[anchor=south]{$~$};

\path (2,-1) node(caption){Fig 3. $n\geq4, 0\leq\theta\leq 1; or \, n=3, 0<\theta\leq1$};  

\draw[->] (8,0) -- (12,0) node[anchor=north] {$\frac{1}{q}$};
\draw[->] (8,0) -- (8,4)  node[anchor=east] {$\frac{1}{r}$};
\path (9.6,-1) node(caption){Fig 4. $n=3,\theta=0$};  

\draw  (8.1, -0.1) node[anchor=east] {$O$};
\draw  (11, 0) node[anchor=north] {$\frac12$};
\draw  (8, 3) node[anchor=east] {$\frac12$};

\draw[thick] (8,3) -- (11,0);  
\draw (7.9,3.15) node[anchor=west] {$A$};
\draw (10.9,0.2) node[anchor=west] {$B$};
\draw (10,2.6) node[anchor=west] {$\frac{2}{q}+\frac{2+\theta}{r}=\frac{2+\theta}2$};

\draw (11,0) circle (0.06);

\filldraw[fill=gray!50](8,3)--(8,0)--(11,0); 

 \draw[red, dashed,thick]         (8.0,0) -- (11,0); 

\draw[<-] (9,2.1) -- (10,2.6) node[anchor=south]{$~$};

\path (6,-1.5) node(caption){Diagrammatic picture of the range of $(q,r)$, when $\nu_0\geq (1-\theta)/(n-1+\theta)$.};  

\end{tikzpicture}

\end{center}

We have thus completed the proof of the homogeneous estimates in \eqref{V-stri}, in all the cases. The corresponding inhomogeneous inequalities  follow from the Christ-Kiselev lemma \cite{CK} as in subsection \ref{sec:inh} as we are assuming $q>2$. This completes the proof of \eqref{V-stri}.\vspace{0.2cm}

\subsection {The Strichartz estimates in \eqref{V-stri'} } In this subsection, we focus on the Strichartz estimates on the line $BC$ in Figure 3 since the other 
Strichartz estimates are proved in the above subsection.\vspace{0.2cm}

By using the Sobolev inequality \eqref{est:sobolev}, the proof of Strichartz estimates \eqref{V-stri'} on the line $BC$ reduces to study the point $B$. We thus prove the following
\begin{proposition}\label{prop:end-stri} Let $n\geq3$ and let $\nu_0>\frac{1-\theta}{n-1+\theta}$ with $0\leq\theta\leq1, (n\geq4)$ or $0<\theta\leq1, (n=3)$. Suppose $u(t,z)$ solves \eqref{equ:KG} with $F=0$. Then there exists a constant $C$ such that
\begin{equation}\label{end-stri}
\begin{split}
\|(1+\Delta_g)^{-\frac{\alpha+\alpha_0}2} &u(t,z)\|_{L^2(\R;L^{\frac{2(n-1+\theta)}{n-3+\theta}}(X))}
\\&\leq C\left(\|u_0\|_{H^{\frac{n+1+\theta}{2(n-1+\theta)}-\alpha-\alpha_0}(X)}+\|u_1\|_{H^{\frac{n+1+\theta}{2(n-1+\theta)}-\alpha-\alpha_0-1}(X)}\right)
\end{split}
\end{equation}
where $\alpha, \alpha_0$ are as in \eqref{V-stri'}.
\end{proposition}

\begin{proof} To prove \eqref{end-stri}, we split $u$ into a low and a high frequencies part. Define $u^l(t,z)=\varphi_0(\sqrt{\LL_V})u(t,z)$ and $u^h(t,z)=(1-\varphi_0)(\sqrt{\LL_V})u(t,z)$, thus $u=u^l+u^h$.\vspace{0.2cm}

{\bf Step 1:} We first consider $u^l$. Since $u^l$ solves the equation
\begin{equation*}
\partial_{t}^2u^l+\LL_V u^l+u^l=0, \quad u(0)=u^l_0,
~\partial_tu(0)=u^l_1,
\end{equation*}
we have by the Duhamel formula
\begin{equation}\label{l-duhamel}
\begin{split}
u^l(t,z)&=\frac{e^{it\sqrt{1+\LL_V}}+e^{-it\sqrt{1+\LL_V}}}2 u^l_0+\frac{e^{it\sqrt{1+\LL_V}}-e^{-it\sqrt{1+\LL_V}}}{2i\sqrt{1+\LL_V}}u^l_1\\&=\frac{e^{it\sqrt{1+\Delta_g}}+e^{-it\sqrt{1+\Delta_g}}}2 u^l_0+\frac{e^{it\sqrt{1+\Delta_g}}-e^{-it\sqrt{1+\Delta_g}}}{2i\sqrt{1+\Delta_g}}u^l_1\\&\quad+\int_0^t\frac{\sin{(t-\tau)\sqrt{1+\Delta_g}}}
{\sqrt{1+\Delta_g}}(V(z)u^l(\tau,z))d\tau.
\end{split}
\end{equation}
Now we aim to prove
\begin{equation}\label{aim:end-l-stri}
\begin{split}
\|(1+\Delta_g)^{-\frac{\alpha+\alpha_0}2} &u^l(t,z)\|_{L^2(\R;L^{\frac{2(n-1+\theta)}{n-3+\theta}}(X))}
\\&\leq C\left(\|u^l_0\|_{H^{\frac{n+1+\theta}{2(n-1+\theta)}-\alpha-\alpha_0}(X)}+\|u^l_1\|_{H^{\frac{n+1+\theta}{2(n-1+\theta)}-\alpha-\alpha_0-1}(X)}\right).
\end{split}
\end{equation}
We need the following
\begin{lemma}\label{lem:end-l-stri} Let $u^l(t,z)=\varphi_0(\sqrt{\LL_V})u(t,z)$, then
\begin{equation}\label{end-l-stri}
\begin{split}
\Big\|(1+\Delta_g)^{-\frac{\alpha+\talpha}2}\int_0^t\frac{\sin{(t-\tau)\sqrt{1+\Delta_g}}}
{\sqrt{1+\Delta_g}}&(V(z)u^l(\tau,z))d\tau\Big\|_{L^2(\R;L^{\frac{2(n-1+\theta)}{n-3+\theta}}(X))}\\&\lesssim \|u^l_0\|_{ H^{s-\alpha-\talpha}(X)}+\|u^l_1\|_{ H^{s-\alpha-\talpha-1}(X)},
\end{split}
\end{equation}
where $s=\frac{n+1+\theta}{2(n-1+\theta)}$ and $\alpha, \talpha$ in \eqref{alpha} with $\theta$ as in Proposition \ref{prop:end-stri} and $\ttheta=1$.
\end{lemma}

We postpone the proof of this lemma to the end of this step.  Once \eqref{end-l-stri} is proved, we have 
\begin{equation}\label{end-stri'}
\begin{split}
\|(1+\Delta_g)^{-\frac{\alpha+\talpha}2}u^l(t,z)\|_{L^2(\R;L^{\frac{2(n-1+\theta)}{n-3+\theta}}(X))}\lesssim \|u^l_0\|_{ H^{s-\alpha-\talpha}(X)}+\|u^l_1\|_{ H^{s-\alpha-\talpha-1}(X)}.
\end{split}
\end{equation}
Indeed, by using \eqref{l-duhamel} and Proposition \ref{Str-L0} with $s=\frac{n+1+\theta}{2(n-1+\theta)}$ and $0\leq \theta\leq 1$, we get
\begin{equation*}
\begin{split}
&\|(1+\Delta_g)^{-\frac{\alpha+\talpha}2}u^l(t,z)\|_{L^2(\R;L^{\frac{2(n-1+\theta)}{n-3+\theta}}(X))}\lesssim \|u^l_0\|_{ H^{s-\alpha-\talpha}(X)}+\|u^l_1\|_{ H^{s-\alpha-\talpha-1}(X)}\\&+\Big\|(1+\Delta_g)^{-\frac{\alpha+\talpha}2}\int_0^t\frac{\sin{(t-\tau)\sqrt{1+\Delta_g}}}
{\sqrt{1+\Delta_g}}(V(z)u^l(\tau,z))d\tau\Big\|_{L^2(\R;L^{\frac{2(n-1+\theta)}{n-3+\theta}}(X))}.
\end{split}
\end{equation*}
This together with \eqref{end-l-stri} gives \eqref{end-stri'}. Since $\talpha=\frac1n<\alpha_0=\frac1{n-1}$, we have 
\begin{equation*}
\begin{split}
\|(1+\Delta_g)^{-\frac{\alpha+\alpha_0}2}u^l(t,z)\|_{L^2(\R;L^{\frac{2(n-1+\theta)}{n-3+\theta}}(X))}\lesssim \|(1+\Delta_g)^{-\frac{\alpha+\talpha}2}u^l(t,z)\|_{L^2(\R;L^{\frac{2(n-1+\theta)}{n-3+\theta}}(X))}.
\end{split}
\end{equation*}
Since the initial data is restricted to low frequencies, \eqref{end-stri'} implies \eqref{aim:end-l-stri}.

\begin{proof}[The proof of Lemma \ref{lem:end-l-stri}] By using \eqref{inh} with $\ttheta=1$, we obtain

\begin{equation*}
\begin{split}
&\Big\|(1+\Delta_g)^{-\frac{\alpha+\talpha}2}\int_0^t\frac{\sin{(t-\tau)\sqrt{1+\Delta_g}}}
{\sqrt{1+\Delta_g}}(V(z) u^l(\tau,z))d\tau\Big\|_{L^2(\R;L^{\frac{2(n-1+\theta)}{n-3+\theta}}(X))}\\
&\lesssim \|V(z) u^l(\tau,z)\|_{L^{2}_t{L}^{\frac{2n}{n+2},2}_z}\lesssim \|r^{-1} u^l(\tau,z)\|_{L^{2}_t{L}^{2}_z}.
\end{split}
\end{equation*}
By using \eqref{l-local-s} of Proposition \ref{prop:loc} with $\beta=1$ that, we prove
\begin{equation*}
\begin{split}
\|r^{-1} \varphi_0(\sqrt{\LL_V})u(t,z)\|_{L^{2}_t{L}^{2}_z}
\lesssim \|\varphi_0(\sqrt{\LL_V}) u_0\|_{ L^{2}(X)}+\|\varphi_0(\sqrt{\LL_V}) u_1\|_{ L^{2}(X)}.
\end{split}
\end{equation*}
Hence this shows \eqref{end-l-stri}. 
\end{proof}

{\bf Step 2:} We now consider $u^h$, that we deal with in a very similar way. By the Duhamel formula we can write
\begin{equation}\label{h-duhamel}
\begin{split}
u^h(t,z)&=\frac{e^{it\sqrt{1+\LL_V}}+e^{-it\sqrt{1+\LL_V}}}2 u^h_0+\frac{e^{it\sqrt{1+\LL_V}}-e^{-it\sqrt{1+\LL_V}}}{2i\sqrt{1+\LL_V}}u^h_1\\&=\frac{e^{it\sqrt{1+\Delta_g}}+e^{-it\sqrt{1+\Delta_g}}}2 u^h_0+\frac{e^{it\sqrt{1+\Delta_g}}-e^{-it\sqrt{1+\Delta_g}}}{2i\sqrt{1+\Delta_g}}u^h_1\\&\quad+\int_0^t\frac{\sin{(t-\tau)\sqrt{1+\Delta_g}}}
{\sqrt{1+\Delta_g}}(V(z)u^h(\tau,z))d\tau.
\end{split}
\end{equation}
Now we aim to prove
\begin{equation}\label{aim:end-h-stri}
\begin{split}
\|(1+\Delta_g)^{-\frac{\alpha+\alpha_0}2} &u^h(t,z)\|_{L^2(\R;L^{\frac{2(n-1+\theta)}{n-3+\theta}}(X))}
\\&\leq C\left(\|u^h_0\|_{H^{\frac{n+1+\theta}{2(n-1+\theta)}-\alpha-\alpha_0}(X)}+\|u^h_1\|_{H^{\frac{n+1+\theta}{2(n-1+\theta)}-\alpha-\alpha_0-1}(X)}\right).
\end{split}
\end{equation}
We thus need to prove the following
\begin{lemma}\label{lem:end-h-stri} Let $u^h(t,z)=(1-\varphi_0)(\sqrt{\LL_V})u(t,z)$, then
\begin{equation}\label{end-h-stri}
\begin{split}
\Big\|(1+\Delta_g)^{-\frac{\alpha+\talpha}2}\int_0^t\frac{\sin{(t-\tau)\sqrt{1+\Delta_g}}}
{\sqrt{1+\Delta_g}}&(V(z)u^h(\tau,z))d\tau\Big\|_{L^2(\R;L^{\frac{2(n-1+\theta)}{n-3+\theta}}(X))}\\&\lesssim \|u^h_0\|_{ H^{s-\alpha-\talpha}(X)}+\|u^h_1\|_{ H^{s-\alpha-\talpha-1}(X)},
\end{split}
\end{equation}
where $s=\frac{n+1+\theta}{2(n-1+\theta)}$ and $\alpha, \talpha$ in \eqref{alpha} where  $\theta$ is in Proposition \ref{prop:end-stri}  and $\ttheta=0$.
\end{lemma}

Once \eqref{end-h-stri} is proved, we have 
\begin{equation}\label{end-h-stri'}
\begin{split}
\|(1+\Delta_g)^{-\frac{\alpha+\talpha}2}u^h(t,z)\|_{L^2(\R;L^{\frac{2(n-1+\theta)}{n-3+\theta}}(X))}\lesssim \|u^h_0\|_{ H^{s-\alpha-\talpha}(X)}+\|u^h_1\|_{ H^{s-\alpha-\talpha-1}(X)}.
\end{split}
\end{equation}
Indeed, by using \eqref{h-duhamel} and Proposition \ref{Str-L0} with $s=\frac{n+1+\theta}{2(n-1+\theta)}$ and $0\leq \theta\leq 1$, we show that
\begin{equation*}
\begin{split}
&\|(1+\Delta_g)^{-\frac{\alpha+\talpha}2}u^h(t,z)\|_{L^2(\R;L^{\frac{2(n-1+\theta)}{n-3+\theta}}(X))}\lesssim \|u^h_0\|_{ H^{s-\alpha-\talpha}(X)}+\|u^h_1\|_{ H^{s-\alpha-\talpha-1}(X)}\\&+\Big\|(1+\Delta_g)^{-\frac{\alpha+\talpha}2}\int_0^t\frac{\sin{(t-\tau)\sqrt{1+\Delta_g}}}
{\sqrt{1+\Delta_g}}(V(z)u^h(\tau,z))d\tau\Big\|_{L^2(\R;L^{\frac{2(n-1+\theta)}{n-3+\theta}}(X))}.
\end{split}
\end{equation*}
This together with \eqref{end-h-stri} gives \eqref{end-h-stri'}. Note that $\alpha_0=\talpha(0)=\frac1{n-1}$,  thus \eqref{end-h-stri'} shows \eqref{aim:end-h-stri}.

\begin{proof}[The proof of Lemma \ref{lem:end-h-stri}] By using \eqref{inh} with $\ttheta=0$, we obtain

\begin{equation*}
\begin{split}
&\Big\|(1+\Delta_g)^{-\frac{\alpha+\talpha}2}\int_0^t\frac{\sin{(t-\tau)\sqrt{1+\Delta_g}}}
{\sqrt{1+\Delta_g}}(V(z) u^h(\tau,z))d\tau\Big\|_{L^2(\R;L^{\frac{2(n-1+\theta)}{n-3+\theta}}(X))}\\
&\lesssim \|V(z) u^h(\tau,z)\|_{L^{2}_t{L}^{\frac{2(n-1)}{n+1},2}_z}\lesssim \|r^{-\frac{n-2}{n-1}}u(\tau,z)\|_{L^{2}_t{L}^{2}_z},
\end{split}
\end{equation*}
By using \eqref{h-local-s} of Proposition \ref{prop:loc} with $\beta=(n-2)/(n-1)$ that, we prove
\begin{equation*}
\begin{split}
\|r^{-\beta} (1-\varphi_0)(\sqrt{\LL_V})u(t,z)\|_{L^{2}_t{L}^{2}_z}
\lesssim \| u^h_0\|_{ H^{\beta-\frac12}(X)}+\|u^h_1\|_{ L^{\beta-\frac32}(X)}.
\end{split}
\end{equation*}
Recall $s=\frac{n+1+\theta}{2(n-1+\theta)}$, thus $\beta-\frac12= s-\alpha-\talpha$, hence this shows \eqref{end-h-stri}. 
\end{proof}
Collecting \eqref{aim:end-l-stri} and \eqref{aim:end-h-stri}, we complete the proof of \eqref{end-stri}, and hence we obtain \eqref{V-stri'}.
\end{proof}

\subsection{The necessity of the assumption \eqref{Ls}  }
We conclude with the following result, to claim the restriction \eqref{Ls} is necessary for Theorem \ref{thm:Strichartz}.

\begin{proposition}\label{count}If $(q,r)\in \Lambda_{s,\theta}$ but  $(q,r)\notin \{(q,r):\frac1 r>\frac12-\frac{1+\nu_0}n\}$ the Strichartz estimates might fail.

\end{proposition}

The proof of this Proposition is identical to the one  in \cite[Proposition 6.2]{ZZ2}. We omit the details.

\begin{center}

\end{center}
\end{document}